\documentclass[12pt]{amsart}
 
\usepackage[margin=1in]{geometry} 
\usepackage{amsmath,amsthm,amssymb,scrextend}
\usepackage{tikz-cd}
\usepackage{comment}
\usepackage{fancyhdr}
\usetikzlibrary{arrows.meta}
\usepackage{enumitem}  
\usepackage{mathabx}
\usetikzlibrary{positioning}
\usetikzlibrary{calc}
\usepackage{verbatim}
\usepackage{tikz}
\usepackage{pgfplots}
\usepackage[mathscr]{euscript}
 \let\mathscr\relax
\usepackage[scr]{rsfso}
\usepackage{multicol}
\usepackage{array}
\usepackage[colorlinks=true, pdfstartview=FitV, linkcolor=blue,
citecolor=blue, urlcolor=blue]{hyperref}

\newtheorem{theorem}{Theorem}

\newtheorem{proposition}[theorem]{Proposition}
\newtheorem{lemma}{Lemma}[theorem]
\newtheorem{definition}{Definition}

\title[a.e. Convergence vs Boundedness]{a.e. Convergence versus Boundedness}

\author{Xinyu Gao}
\address{Department of Mathematics, University of Missouri, Columbia, MO 65211, USA}
\email{xggh8@umsystem.edu}

\author{Loukas Grafakos}
\address{Department of Mathematics, University of Missouri, Columbia, MO 65211, USA}
\email{grafakosl@umsystem.edu}

\thanks{The authors acknowledge the support of the Simons Foundation and the University of Missouri Curators Fund.}

\date{}

\begin{document}

\begin{abstract}  
In this paper we extend  Stein's maximal theorem \cite{Stein1961} 
to the bilinear setting; precisely, let $M$ be a homogeneous space on which a compact abelian group acts transitively and let \( 1 \le p, q \le 2 \), \( 1/2 \le r \le 1 \) be related by \( 1/p + 1/q = 1/r \). 
Given a  family of bounded bilinear operators
$T_m : L^p(M) \times L^q(M) \to L^r(M)$, $ m=1,2,\dots$, 
that commute with translations and converge a.e. as $m\to \infty$, we show that  the
associated maximal operator $T^*(f, g) = \sup_m |T_m(f, g)|$  satisfies a weak-type $L^p(M)\times L^q(M) \to L^{r,\infty}(M)$ estimate. Our proof is based on probabilistic arguments, properties of the Rademacher functions, and measure-theoretic constructions.  The main obstacle we overcome 
is an extension of Stein's lemma that provides an estimate for the $L^2$ norm of a tail of a double 
Rademacher series in terms of the $L^\infty$ norm of the series over a measurable set. 

We also obtain a bilinear analogue of Sawyer's \cite{Sawyer1966} extension of Stein’s theorem to the wider range $1\le p,q<\infty$. This result is valid for positive bilinear operators that commute with a mixing family of measure-preserving transformations.

 Our main  application concerns the boundedness of a maximal  bilinear tail operator associated with an ergodic  
measure--preserving transformation on a finite measure space, which was shown  by Assani and
Buczolich~\cite{Assani2010} to be finite a.e..  Our results imply that  
 this maximal operator is bounded from $L^p\times L^q$ to   $L^r$ for the natural exponent
\( r=(\tfrac1p+\tfrac1q)^{-1} \) when $p,q>1$;  this extends the integrability of
the bilinear maximal tail operator from the subcritical region to the ``natural" boundedness region. 
 We discuss additional applications related to the a.e. convergence of 
bilinear Bochner-Riesz means and other bilinear averages on the torus.
\end{abstract}

\maketitle

\section{Introduction} \label{Intro}
A well-known classical fact is that the $L^p$ boundedness of a maximal family of linear operators together with pointwise convergence for a dense subspace of the domain, implies a.e. convergence for all functions on the domain. In contrast to this classical direction, our results show that a.e. convergence of a bilinear sequence forces weak-type bounds for the associated maximal operator. Motivated by the work of Stein \cite{Stein1961} and Sawyer \cite{Sawyer1966}, we obtain bilinear analogues of their results and we also study some of their consequences in terms of applications. 
The main results of this paper are Theorems~\ref{thm1} and \ref{thm2}, while our main application is 
Theorem~\ref{App1}.

In Stein's celebrated work~\cite{Stein1961}, it is  shown that if a sequence of
translation-invariant operators $\{T_m\}$ on a compact group  
acts boundedly on $L^p(M)$, $1\le p\le2$, and if
$T_m f \to T f$ pointwise for every $f\in L^p(M)$, then the maximal operator
\[
T^*f=\sup_{m\ge1}|T_m f|
\]
satisfies a weak-type $(p,p)$ inequality.
The argument is based on probabilistic techniques with Rademacher functions and provides 
a powerful  connection between a.e. convergence and weak-type $(p, p)$ inequalities. 
In this work we develop an  analogue 
for  sequences of bilinear operators $T_m:\,\, L^p(M)\times L^q(M) \to L^r(M)$
with $1/p+1/q=1/r$ and $1\le p,q\le 2$. It should be noted that the opposite direction 
also holds: if a sequence of multilinear   operators satisfies a maximal weak-type estimate 
$L^p\times L^q \to L^{r,\infty}$,  
then it converges a.e. for all $L^p\times L^q$ functions, assuming it does so for a dense subclass; for a proof of this fact see for instance \cite[Proposition 2]{DosidisGrafakos2021}.

Interest in  bilinear operators originated in the pioneering work of 
Coifman and Meyer~\cite{CoifmanMeyer1975}~\cite{CoifmanMeyer1978} in the seventies 
and by the celebrated  work of
Lacey and Thiele~\cite{LaceyThiele1997,LaceyThiele1999}
on the bilinear Hilbert transform in the nineties. 
These works have spurred a resurgence of activity on adaptations  of many 
classical linear results to multilinear analogues, such as the Calder\'on-Zygmund theory
 \cite{GrafakosTorres2002} and many other topics. In relation  to 
 a.e. convergence of multilinear singular operators and boundedness of the 
 associated maximal operators (which naturally implies a.e. convergence) we refer the
 reader to 
 \cite{Lacey}
 \cite{MuscaluTaoThiele2002},  
 \cite{GT2002}
\cite{MuscaluTaoThiele2004}, 
 \cite{LiMuscalu2007},
  \cite{DemeterTaoThiele2008},
  \cite{DoThiele2016}, 
  \cite{MirekSteinTrojan2017}, 
  \cite{GrafakosHeHonzikPark2024}, 
  \cite{DouglasGrafakos2025},
  \cite{KMTT}. This list is by no means exhaustive, but is representative of the work in this area.

Another contribution of this work is the bilinear analogue of Sawyer's extension of Stein's theorem
to the case $r>2$ for positive operators under certain mild mixing conditions. 
The main  application of our results concerns improved weak-type bounds for a
 bilinear tail operator associated with an ergodic measure-preserving transformation on a finite measure space.

We begin by  introducing some preliminaries and then develop the probabilistic and measure-theoretic framework necessary for establishing Theorem~\ref{thm2}.

Let \( G \) be a topological group and let \( M \) be a topological space. We say that \( G \) 
acts continuously on $M$ if there is a continuous  mapping:
\[
G \times M  \to M, \qquad (g,x) \mapsto g(x).
\]
We call \( M \) a \emph{homogeneous space} of \( G \) if $G$ acts continuously on $M$ and 
transitively; the latter means for every pair of points \( x, y \in M \) there exists an element \( g \in G \) such that
\[
g(x) = y.
\]
We also define the translation operator \( \tau_g \) acting on functions $f$ on $M$ by 
\[
(\tau_g f)(x) = f(g^{-1}x), \qquad g \in G.
\]
Now suppose \( G \) is a compact group and \( M \) is a homogeneous space of \( G \).
Then the homogeneous space \( M \) inherits a unique normalized \( G \)-invariant measure \( d\mu \) from \( G \) that satisfies 
\[
\int_M f(g(x))\,d\mu(x) = \int_M f(x)\,d\mu(x),
\qquad \text{for all } f \in L^1(M) \text{ and } g \in G.
\]

If we normalize the measure \( d\mu \) on \( M \) such that 
\[
\int_M d\mu = 1,
\]
then \( G \) also has {\it a finite Haar measure} \( d\omega_G \), that satisfies
\[
\int_G d\omega_G = 1.
\]

It is useful to recall the relation between the measures \( d\mu \) on \( M \) and \( d\omega_G \) on \( G \).
For a fixed point  \( x_0 \in M \) and     a Borel subset \( E \) of \( M \) 
we define the {\it fiber of $E$ over the point $x_0$  under the action of $G$} as follows: 
\[
\widehat{E} = \{ g \in G : g(x_0) \in E \}.
\]
Then, by translation invariance, for any $x_0 \in M$ we have 
\begin{equation}\label{mu}
\mu(E) =\omega_G(\widehat{E}).
\end{equation}

Some classical examples that the reader may keep in mind are the following.  
First, when both the group and the space coincide with the \(k\)-torus, \(G = M = \mathbb{T}^k\),
the action is simply translation modulo 1:
\[
g \cdot x = g + x \pmod{1}, \qquad g, x \in \mathbb{T}^k,
\]
that is, addition is taken componentwise on the torus.

A second, closely related example arises when \(G = M\) is the {\it dyadic group},
in which case the action is given by componentwise dyadic addition:
\[
g \cdot x = g + x \pmod{2}.
\]

Finally, one may think of the rotational setting:
let \(G = \mathrm{SO}(n)\), the group of rotations in \(\mathbb{R}^n\),
and \(M = S^{n-1}\), the unit sphere.
Here the natural action is the usual rotation of vectors,
\[
g \cdot x = g(x), \qquad g \in \mathrm{SO}(n), \; x \in S^{n-1}.
\]
These three cases serve as   canonical examples 
of compact group actions and motivate the general framework discussed below.

We now discuss the setup in the bilinear setting.  Let $M$ and $G$ be as above.

\begin{definition}
For each $m=1,2,\dots$, let $T_m$ be a bilinear operator defined on $L^p(M, d\mu)\times 
L^q(M, d\mu)$, where $1\le p,q\le 2$.

\begin{itemize}
    \item We say that each $T_m$ is bounded if  for each $m=1,2,\dots$ there is a constant $C_m$ such that 
    \[
    \|T_m(f, h)\|_{L^r(M)} \leq C_m \|f\|_{L^p(M)} \|h\|_{L^q(M)},
    \]
    for all $f\in L^p(M)$ and $h \in L^q(M)$.
    We do not assume the constants $C_m$ are uniform in $m$.
    \item We say that each $T_m$ commutes with simultaneous translations if for every $g \in G$ we have 
    \[
    T_m(\tau_g f, \tau_g h)(x) = \tau_g T_m(f, h)(x), \quad \forall \, x \in M.
    \]

\item  
For $f\in L^p(M)$ and $h\in L^q(M)$ we define the associated maximal operator: 
\[
T^*(f, h) := \sup_{m \geq 1} |T_m(f, h) |.
\]
\end{itemize}
\end{definition}

The  first main result of this work is the following theorem:

\begin{theorem} \label{thm1}
Let $1 \le p,q \leq 2$ and $\frac{1}{2} \leq r \leq 1$, such that $\frac{1}{r} = \frac{1}{p} + \frac{1}{q}$. Let $T_m$ be a sequence of bounded operators  from  $ L^p(M) \times L^q(M)$  to $ L^r(M)$
that commute  with simultaneous translations. 
Suppose that for every $f \in L^p(M)$, $h \in L^q(M)$, the pointwise limit
    \[
    \lim_{m  \to \infty} T_m(f, h)(x)
    \]
    exists for almost every $x \in M$.

Then, there exists a constant $C > 0$ such that for all $f \in L^p(M)$, $h \in L^q(M)$, and all $\alpha > 0$,
\begin{equation}\label{con}
\mu\left(\left\{ x \in M : T^*(f, h)(x) > \alpha \right\} \right) \leq \frac{C}{\alpha^r} \|f\|_{L^p(M)}^r \|h\|_{L^q(M)}^r.
\end{equation}
\end{theorem}

In Section \ref{Sawyer} we also prove a version of this result for positive operators. In Section \ref{APP} we discuss applications. Naturally, our results extend to the case of multilinear operators, but for the 
sake of notational simplicity we only focus on the bilinear case.

\medskip
{\bf Notation:} We denote Lebesgue measure on Euclidean space or the torus by $| \cdot  |$. The indicator function of a set $S$ is denoted by $\mathbf 1_S$. We use the notation $A\lesssim B$ to indicate that the quantity $A$ is controlled by a constant multiple of $B$, for some 
inessential constant. 

\section{Some preliminary results}

We begin the proof with some lemmas. The following summability lemma allows us to balance size and frequency of occurrence of the terms of a series. 

\begin{lemma}\label{lm21}
Let $(a_n)_{n=1}^\infty$ be a sequence of numbers satisfying $0<a_n<\frac{A}{n}$ for some $A>0$. Then   we can find  a sequence of natural numbers
$n_1\le n_2\le n_3\le \cdots$ such that:
\begin{equation}\label{sum1}
    \sum_{k=1}^\infty a_{n_k}<\infty,
\end{equation}
\begin{equation}\label{sum2}
\sum_{k=1}^\infty n_ka_{n_k}=\infty.
\end{equation}
\end{lemma}
\begin{proof}
If $\limsup_{n \to \infty} na_n=a>0$, we define 
$n_1=1$ and for $k\ge 2$ we define inductively
$$
n_k= \min\Big\{n>\max(k^2, n_{k-1}):\,\, na_n>\frac{a}{k+1}\Big\}.
$$ 
Then we have   
$$\sum_{k=1}^\infty a_{n_k}\le\sum_{k=1}^\infty \frac{A}{k^2}<\infty,$$
$$\sum_{k=1}^K n_ka_{n_k}>a\sum_{k=1}^K\frac{1}{k+1}.$$

If now $\limsup_{n \to \infty} na_n=0$, then we define $(a_{n_k})$ to be the following sequence:
$$\underbrace{a_1,\dots,a_1}_{\lfloor\frac{1}{2a_{2}}\rfloor \,\,\textup{times}},\underbrace{a_2,\dots,a_2}_{\lfloor\frac{1}{2^2a_{2^2}}\rfloor \,\,\textup{times}},\dots,\underbrace{a_k,\dots,a_k}_{\lfloor\frac{1}{2^ka_{2^k}}\rfloor \,\,\textup{times}},\dots .
$$
Here  $\lfloor a \rfloor = \max\{ n \in \mathbb{Z} : n \le a \}$ denotes the floor of $a$, i.e., 
the greatest integer less than or equal to \(a\).
Notice that $\lim_{t \to \infty}\lfloor\frac{1}{2^ta_{2^t}}\rfloor 2^ta_{2^t}=1$, and $\lfloor\frac{1}{ 2^ta_{2^t}}\rfloor a_{2^t}<2^{-t}$, for every $t\in\mathbb N$, therefore, \eqref{sum1} and \eqref{sum2} hold. 
\end{proof}

\begin{lemma}\label{lm22}
For a positive numerical series $\sum_{n=1}^\infty a_n<\infty$, we can find a sequence  of positive numbers $R_n$ such that $R_n \to \infty$ as $n \to \infty$, but also such that $\sum_{n=1}^\infty R_na_n<\infty$.
\end{lemma}

\begin{proof}
Denote by $T_n=\sum_{k=n+1}^\infty a_k$ the $n$th tail of $\sum_{n=1}^\infty a_n$.  Then take $R_n=\frac{1}{\sqrt{T_n}}$. It is straightforward to verify that $R_n \to \infty$, and
{\allowdisplaybreaks\begin{align*}
\sum_{n=1}^N R_na_n=&\sum_{n=1}^N\frac{a_n}{\sqrt{T_n}}\\
=&\sum_{n=1}^N\frac{T_{n-1}-T_n}{\sqrt{T_n}}\\
=&\sum_{n=1}^N\sqrt{T_n}\bigg(1-\frac{T_n}{T_{n-1}}\bigg)\\
\le&\sum_{n=1}^N\frac{\sqrt{T_{n-1}}}{2}\bigg(1-\frac{\sqrt{T_n}}{\sqrt{T_{n-1}}}\bigg)\\
=&\frac{1}{2}\sum_{n=1}^N \sqrt{T_{n-1}}-\sqrt{T_n}\\
=&\frac{1}{2} \big(\sqrt{T_0}-\sqrt{T_N}\big)
\end{align*}}
As $T_0=\sum_{n=1}^\infty a_n<\infty$ and $\lim_{N \to \infty} T_N=0$, we complete the proof.
\end{proof}
Then by the above lemma, we have the following
\begin{lemma}\label{lm31} Let  $0<p,q <\infty$.
Let \( \{a_n\}_{n=1}^\infty \), \( \{b_n\}_{n=1}^\infty \) be sequences of positive numbers such that
\[
\sum_{n=1}^\infty a_n^p<\infty,\quad\sum_{n=1}^\infty b_n^q < \infty.
\]
Then, there exists a sequence \( \{R_n\}_{n=1}^\infty \) with \( R_n \to \infty \) such that
\[
\sum_{n=1}^\infty R_n^p a_n^p < \infty, \quad \sum_{n=1}^\infty R_n^q b_n^q < \infty.
\]
\end{lemma}
\begin{proof}
Take $c_n=a_n^p+b_n^q$, then by Lemma~\ref{lm22}, there exists a sequence \( \{R_n'\}_{n=1}^\infty \) with \( R_n' \to \infty \) such that
\[
\sum_{n=1}^\infty R_n'c_n < \infty.
\]
Setting $R_n'=R_n^{\max\{p,q\}}$, we rewrite the preceding as  
\[
\sum_{n=1}^\infty R_n^{\max\{p,q\}}c_n < \infty. 
\]
Then we have that 
\[ \sum_{n=1}^\infty R_n^p a_n^p < \sum_{n=1}^\infty R_n^{\max\{p,q\}}c_n<\infty,\]
\[ \sum_{n=1}^\infty R_n^q b_n^q <\sum_{n=1}^\infty R_n^{\max\{p,q\}}c_n< \infty.\]
This proves our claim.
\end{proof}

\begin{lemma}\label{ln}
Let $(t_n)_{n\ge1}$ be a sequence of real numbers such that
\[
0 < t_n < 1 \quad \text{for all } n=1,2,\dots
\]
and
\[
\sum_{n=1}^\infty t_n = \infty.
\]
Then, for every fixed natural number \(N\), we have
\[
\prod_{n=N}^M (1 - t_n) \longrightarrow 0
\quad \text{as } M \to \infty.
\]
\end{lemma}

\begin{proof}
Fix \(N\in\mathbb N\) and set
\[
P_M := \prod_{n=N}^M (1 - t_n), \qquad M \ge N.
\]
Since \(0 < t_n < 1\), each factor \(1 - t_n\) lies in \((0,1)\), so \(P_M>0\) and
\[
\ln P_M = \sum_{n=N}^M \ln(1 - t_n).
\]

We now estimate the logarithms. For \(x\in(0,1)\), we have the standard inequality
\[
\ln(1-x) \le -x.
\]

Applying this with \(x = t_n\), we obtain
\[
\ln(1 - t_n) \le -t_n \quad\text{for all } n.
\]
Therefore,
\[
\ln P_M
=
\sum_{n=N}^M \ln(1 - t_n)
\;\le\;
- \sum_{n=N}^M t_n,
\]
thus 
\[
P_M
=e^{\sum_{n=N}^M \ln(1 - t_n)}
\;\le\;
 e^{-\sum_{n=N}^M t_n},
\]

By hypothesis, \(\sum_{n=1}^\infty t_n = \infty\), hence also
\(\sum_{n=N}^\infty t_n = \infty\). Thus
\[
\sum_{n=N}^M t_n \xrightarrow[M \to\infty]{} \infty,
\]
which implies
$P_M \rightarrow 0
$ as $M \to\infty$. 
This is exactly the desired conclusion.
\end{proof}

\begin{lemma}\label{lm24}
Let $E_1,E_2, \dots,E_n,\dots$ be a collection of sets in $M$, with the property that 
$\sum_{n=1}^\infty \mu(E_n) = \infty$. Then there exists a sequence of elements
$g_1,g_2,\dots, g_n,\dots$ belonging to $G$, so that the translated sets $F_1,F_2,\dots$  defined by 
$F_n = g_n[E_n]$, have the   property that almost every point of $M$ belongs to infinitely many sets $F_n$. Precisely, let
$$F_0 = \limsup F_n= \bigcap_{k=1}^\infty\bigcup_{n=k}^\infty F_n.$$
Then $\mu(F_0) = 1$.  
\end{lemma}
\begin{proof}
We consider two infinite product spaces. First, let 
$$
\mathcal M=\prod_{k=1}^\infty M_k,
$$ 
where $\mathcal M$ is the infinite product of $M_k$'s, each $M_k$ being a copy of $M$. Thus the points of $\mathcal M$ are sequences $\{x_n\}$, where $x_n \in M$. We impose the usual product measure of the measures $d\mu$ on each $M_k$ on $\mathcal M$. We call this product measure $d\mu^*$. 

Next we consider the infinite product group 
$$
\Gamma = \prod_{k=1}^\infty  G_k,
$$
where $\Gamma$ is the infinite product of $G_k$, each
$G_k$ is a copy of $G$. The elements of $\Gamma$ are sequences $\{g_n\}$, where $g_n \in G$. On
$\Gamma$ we also consider the usual product measure. 

We notice that $\Gamma$ is a compact group which acts on $\mathcal M$ in a the following way
\[
\gamma \cdot (x_1,x_2,\dots) = (g_1(x_1), g_2(x_2), \dots), 
\qquad \gamma=(g_1,g_2,\dots)\in\Gamma.
\]
Hence $(\mathcal{M},\mu^*)$ is a homogeneous space of~$\Gamma$.

Consider now the collection of sets $\mathcal E_1$, $\mathcal E_2$,$\dots$, $\mathcal E_n$, $\dots$ on $\mathcal M$ defined as follows
\[
\mathcal{E}_n = \{\, x=(x_1,x_2,\dots) \in \mathcal{M} : \,\, x_n \in E_n \,\}, \qquad n=1,2,\dots. 
\]
 Then $\mu^*(\mathcal E_n) = \mu(E_n)$. 
 Let 
 $$
 \mathcal E_0 = \bigcap_{k=1}^\infty\bigcup_{n=k}^\infty \mathcal E_n.
 $$
We claim that \( \mu^*(\mathcal E_0) = 1 \).

Recall that
\[
\mathcal E_0^\complement
= \bigcup_{k=1}^\infty \bigcap_{n=k}^\infty \mathcal E_n^\complement.
\]
Hence, by the subadditivity of the outer measure \( \mu^* \)  we have 
\begin{equation}\label{eq:union-bound}
\mu^*(\mathcal E_0^\complement)
\le
\sum_{k=1}^\infty \mu^*\!\left(\bigcap_{n=k}^\infty \mathcal E_n^\complement\right).
\end{equation}

Fix \( k \ge 1 \).  
For any \( m > k \),
\[
\bigcap_{n=k}^m \mathcal E_n^\complement
\subseteq
\bigcup_{n=k}^m \mathcal E_n^\complement,
\]
so by the monotonicity of \( \mu^* \) one obtains 
\begin{equation*}
\mu^*\!\left(\bigcap_{n=k}^m \mathcal E_n^\complement\right)
=
\prod_{n=k}^m \mu(  E_n^\complement)
= \prod_{n=k}^m (1-\mu(  E_n)).
\end{equation*}
Since, by assumption, \( \mu^*(\mathcal E_n) = \mu(E_n)  > 0 \),
we have \( \mu^*(\mathcal E_n^\complement) = 1 - \mu(E_n) \) 
by the definition of the product measure. Thus,
\[
\mu^*\!\left(\bigcap_{n=k}^m \mathcal E_n^\complement\right)
\le  \prod_{n=k}^m (1-\mu(  E_n)).
\]
Letting \( m \to \infty \), we obtain
\begin{equation}\label{eq:limit-zero}
\mu^*\!\left(\bigcap_{n=k}^\infty \mathcal E_n^\complement\right)
\le \lim_{m\to\infty}  \prod_{n=k}^m (1-\mu(  E_n))
= 0,
\end{equation}
by Lemma~\ref{ln}.

Now summing over \( k \) and substituting \eqref{eq:limit-zero} into \eqref{eq:union-bound}, we find that 
\[
\mu^*(\mathcal E_0^\complement)
\le
\sum_{k=1}^\infty
\mu^*\!\left(\bigcap_{n=k}^\infty \mathcal E_n^\complement\right)
= 0 
\]
 Hence,
\[
\mu^*(\mathcal E_0)
= 1 - \mu^*(\mathcal E_0^\complement)
= 1.
\]
Therefore we have $\mu(\mathcal E_0^\complement) = 0$  and thus $\mu(\mathcal E_0) = 1$.

Let now $\psi(x_1, x_2,\dots,x_n,\dots)$   be the characteristic function of the set $\mathcal E_0$.
Consider the function $f(\gamma, p_x)$ defined on $\Gamma
\times M$ as follows
$$
f(\gamma, x) = \psi(g_1^{-1}(x), g_2^{-1}(x),\dots, g_n^{-1}(x) ,\dots),
$$
where $\gamma = \{g_n\} \in \Gamma$, and $x\in M$. 

Define the set
\[
A_x := \{\gamma\in\Gamma : f(\gamma,x)=1\}
      = \{\gamma\in\Gamma : \gamma^{-1}\cdot p_x \in \mathcal E_0\}.
\]
We apply the \eqref{mu} to the case
where $\Gamma$ is the group, $\mathcal M$ the homogeneous space 
$$
p_x = (x, x,\dots x,\dots) ,
$$
 and $E = \mathcal E_0$, we obtain
\[
\mu^*(\mathcal E_0)
  = \omega_\Gamma(A_x),
\]
where $ \omega_\Gamma$ is the product Haar measure induced by $\omega_G$. 
Since $\mu^*(\mathcal E_0)=1$, it follows that for every fixed $x\in M$,
\[
\omega_\Gamma(A_x)=1,
\quad\text{i.e.}\quad
f(\gamma,p_x)=1\ \text{for almost every }\gamma\in\Gamma.
\]

 We then have that for each $x$, $f(\gamma, p_x) = 1$ for almost every $\gamma$. Hence, by Fubini's theorem
 $$
 \int_{M}\int_{\Gamma}f(\gamma, p_x)\, d\omega_\Gamma \,d\mu^*= 1=
 \int_{\Gamma}\int_{M}f(\gamma, p_x)
\, d\mu^* \,d\omega_\Gamma .
 $$ 
 Thus for almost every  $\gamma = \{g_n\}\in \Gamma$, we have $f(\gamma, p_x)=1$ for almost every $x\in M$. 

Therefore for almost every $x \in M$, $\gamma^{-1}(p_x) \in \mathcal E_0$. That
is, for almost every $x \in M$, $g_n^{-1}(x) \in E_n$ for infinitely many $n$. Hence, for
almost every $x \in M$, we have $x \in g_n[E_n]$ for infinitely many $n$. This proves the
lemma.
\end{proof}

Next we consider the \textbf{Rademacher} functions $r_n(t)$ defined by 
\[
 r_n(t) = r_1(2^n t),\quad n=1,2,\dots,
\]
for $t\in I:=[0,1]$, where $r_0(t)=1$ and
\[
r_1(t) = 
\begin{cases}
1, & \text{if } 0 \le t \le 1/2, \\
-1, & \text{if } 1/2 < t < 1.
\end{cases}
\]

These functions are orthonormal over $[0, 1]$ and their importance lies in  the fact that they can be thought of as mutually independent random variables. We shall not make explicit use of this fact, but instead we shall use a property that, in effect, follows from this.

Our next lemma   is  an analogue of Lemma 2 in \cite{Stein1961} (see also \cite{SagherZhou}) for double Rademacher series.  
The approach in the proof below    is based on the work of 
\cite{AGH}. Let $r_n(s)$, $r_k(t)$ be the Rademacher functions indexed by $k,n\ge 0$, and for  numbers $a_{n,k}$ with $\sum_{n,k=0}^\infty |a_{n,k}|^2<\infty$, we consider the double 
Rademacher series
$$
F(s,t) = \sum_{n,k=0}^\infty a_{n,k} \, r_n(s)\, r_k(t). 
$$

The following lemma concerns $F$. 

\begin{lemma}\label{lm33}
Let $E\subset[0,1]^2$ be a measurable set with $|E|>0$, and let
\[
F(s,t)=\sum_{n,k=0}^\infty a_{n,k}\, r_n(s)\,r_k(t),
\qquad \sum_{n,k\ge0}|a_{n,k}|^2<\infty,
\]
where $\{a_{n,k}\}_{n,k=0}^\infty$ is a sequence of complex numbers.
Then there exists an integer $N=N(E)$ and a constant $A=A(E)>0$, both depending only on $E$ and not on the coefficients $a_{n,k}$,
 such that 
\[
\Big(\sum_{\substack{n>N(E)\\ k>N(E)}} |a_{n,k}|^2\Big)^{1/2}
\ \le\
A(E)\,\operatorname*{ess\,sup}_{(s,t)\in E}|F(s,t)|.
\]
\end{lemma}

\begin{proof}
Pick an $\varepsilon>0$ small enough such that
\[
0<\frac{1}{\tfrac12-\varepsilon-\sqrt{\varepsilon}} < 2.
\]

By standard measure theory, one can show that for $E\subset[0,1]^2$, for any $\varepsilon>0$, there exists a dyadic rectangle of size $N=N(E)$
\[
R_N = I_N\times J_N, \qquad |I_N|=|J_N|=2^{-N},
\]
such that
\[
|E\cap R_N|\ \ge\ (1-\varepsilon)\,|R_N|.
\]
This number $N$   depends only on $E$ and this is   $N(E)$ claimed in the lemma.

Define
\[
H_{00}(s,t) := \sum_{\substack{0\le n\le N\\0\le k\le N}} a_{n,k}\,r_n(s)\,r_k(t),
\qquad
F_N(s,t) := \sum_{n>N\ \text{or}\ k>N} a_{n,k}\,r_n(s)\,r_k(t),
\]
so that $F = H_{00}+F_N$.

We split the index set
\[
\{(n,k): n>N\ \text{or}\ k>N\}
= A_1\cup A_2\cup A_3,
\]
where
\[
A_1=\{n>N,\ k>N\},\qquad
A_2=\{n>N,\ 0\le k\le N\},\qquad
A_3=\{0\le n\le N,\ k>N\},
\]
and we also split $F_N$ accordingly as 
\[
F_N = F^{(1)}+F^{(2)}+F^{(3)},
\]
where in each $F^{(i)}$ the indices $(n,k)$ range over the set $A_i$, $i=1,2,3$. Notice that $A_1$, $A_2$, $A_3$ are pairwise disjoint sets. 

The functions $\{r_n(s)r_k(t)\}_{n,k\ge 0}$ form an orthonormal system on
$[0,1]^2$. 
 For every $n,k>N$, the functions $r_n$ and $r_k$ take the values
$\pm 1$ on subsets of equal measure of $I_N$ and $J_N$, respectively; 
consequently, the   family
\[
\{\, r_n(s)\, r_k(t): n,k>N \,\}
\]
is orthogonal on the rectangle $R_N$. 

Thus
\begin{equation}\label{F1}
\iint_{R_N} |F^{(1)}(s, t)|^2 ds \, dt
= |R_N|\sum_{\substack{n>N\\k>N}} |a_{n,k}|^2.
\end{equation}

But when $k\le N$ we have that $r_k(t)$    is constant
on~$J_N$.  Writing $r_k(t)=\beta_k=\pm1$ for this constant value, we have
\begin{equation}\label{F2E}
F^{(2)}(s,t)
= \sum_{\substack{n>N\\ k\le N}} a_{n,k}\,\beta_k\,r_n(s),
\end{equation}
and the functions $\{r_n:n>N\}$ remain orthogonal on~$I_N$.

The same argument gives
\begin{equation}\label{F3E}
F^{(3)}(s,t)
= \sum_{\substack{0\le n\le N\\ k> N}} a_{n,k}\,   \alpha_n\,r_k(t), 
\end{equation}
where   $\alpha_n=r_n(t)=\pm 1$  is constant.

By \eqref{F2E}, we have that
\begin{align*}
\langle F^{(1)}, F^{(2)} \rangle_{L^2(R_N)}
&=
\iint_{R_N} 
\bigg(\sum_{(n,k)\in A_1} a_{n,k} r_n(s) r_k(t)\bigg)
\overline{
\bigg(\sum_{\substack{n'>N\\ 0\le k'\le N}} a_{n',k'}\,\beta_{k'}\,r_{n'}(s)\bigg)
}\, ds\,dt \\
&=\int_{I_N} 
 \sum_{k>N}\sum_{n>N}a_{n,k} r_n(s)\int_{J_N}r_k(t)\, dt \, \overline{
\bigg(\sum_{\substack{n'>N\\ 0\le k'\le N}} a_{n',k'}\,\beta_{k'}\,r_{n'}(s)\bigg)
} 
\, ds. 
\end{align*}
By the fact that $\int_{J_N} r_k(t)\,dt=0$ for $k>N$ we conclude that 
\[
    \langle F^{(1)}, F^{(2)} \rangle_{L^2(R_N)}=0.
\]
Similarly, we have that 
\[
    \langle F^{(1)}, F^{(3)} \rangle_{L^2(R_N)}=0.
\]

Then we show that  $\langle F^{(2)}, F^{(3)} \rangle_{L^2(R_N)}=0$. Using \eqref{F2E} and \eqref{F3E}, we have
\begin{align*}
\hspace{-.2in}&\langle F^{(2)}, F^{(3)} \rangle_{L^2(R_N)} \\
&=
\iint_{R_N} 
\bigg(\sum_{\substack{n>N\\ k\le N}} a_{n,k}\,\beta_k\,r_n(s)\bigg)
\overline{
\bigg(\sum_{\substack{0\le n'\le N\\ k'> N}} a_{n',k'}\, \alpha_{n'}\,r_{k'}(t) \bigg)
}\, ds\,dt \\
&= 
\bigg(\sum_{n>N} \Big(\sum_{0\le k\le N}a_{n,k}   \beta_k \Big)\int_{I_N}r_n(s)\,ds\bigg)
\overline{
\bigg( \sum_{k'>N}\Big(\sum_{0\le n'\le N} a_{n',k'}   \alpha_{n'} \Big) \int_{J_N}r_{k'}(t)\,dt\bigg)
}
\end{align*}
Then by the fact that $\int_{I_N}r_n(s)\,ds=0$ and $\int_{J_N} r_{k'}(t)\,dt=0$ for $n,k'>N$, we deduce
\[
\langle F^{(2)}, F^{(3)} \rangle_{L^2(R_N)}=0.
\]

 Therefore, for $i\neq j$, we have
\[
\langle F^{(i)}, F^{(j)}\rangle_{L^2(R_N)} = 0,
\qquad i\ne j.
\]
That is, $F^{(i)}$ and $F^{(j)}$ are orthogonal on~$R_N$.  Hence
\[
\iint_{R_N} |F_N(s, t)|^2 ds\, dt
= \sum_{i=1}^3 \iint_{R_N} |F^{(i)}(s, t)|^2 ds\, dt.
\]

Then by \eqref{F1}, we have the correct bound
\begin{equation}\label{boundF}
|R_N|\sum_{\substack{n>N\\k>N}} |a_{n,k}|^2\le\iint_{R_N} |F_N(s, t) |^2 ds \, dt.
\end{equation}

We also have for any measurable set $E'\subset[0,1]^2$
\begin{align*}
&\int_{E'} \Bigl|\sum_{n,k\ge 0} a_{n,k} r_n(s)r_k(t)\Bigr|^2\,ds\,dt\\
&\le
|E'|\sum_{n,k=0}^\infty |a_{n,k}|^2
+
\sum_{(n,k)\neq (n',k')} a_{n,k} \overline{a_{n',k'}}
\int_{E'} r_n(s)r_k(t)r_{n'}(s)r_{k'}(t)\,ds\,dt\\
&\le |E'|\sum_{n,k=0}^\infty |a_{n,k}|^2
+
\Biggl(
   \sum_{(n,k)\neq (n',k')} |a_{n,k} \overline{a_{n',k'}}|^2
\!\Biggr)^{\!\frac12}
\!\!
\Biggl(
   \sum_{(n,k)\neq (n',k')}
   \Bigl|\int_{E'} r_n(s)r_k(t)r_{n'}(s)r_{k'}(t)\,ds\,dt\Bigr|^2 \!
\Biggr)^{\!\frac12}\\
&\le
|E'|\sum_{n,k=0}^\infty |a_{n,k}|^2
+
\Biggl(\sum_{n,k=0}^\infty |a_{n,k}|^2 \!\Biggr)^{\!\frac12}\! 
\Biggl(
   \sum_{\substack{(n,k)\ne(n',k')}} |\langle \chi_{E'}, r_nr_k \cdot r_{n'} r_{k'} \rangle|^2
\Biggr)^{\!\frac12}
\end{align*}
using the inequality
\[
   \sum_{\substack{n,k,n',k'\ge 0 \\ (n,k)\ne (n',k')}}
   |\langle f, r_nr_k \cdot r_{n'} r_{k'}\rangle|^2
   \le \|f\|_{L^2}^2
\]
for all $f\in L^2([0,1]^2)$., thus we have that 
\begin{equation}\label{sqrt}
   \int_{E'} \Bigl|\sum_{n,k\ge 0} a_{n,k} r_n(s)r_k(t)\Bigr|^2\,ds\,dt \le \Big(|E'|+\sqrt{|E'|}\Big)
 \sum_{n,k\ge0} |a_{n,k}|^2.
\end{equation}

We now define sets 
\[
R_N^{++}
= \biggl\{\, (s,t)\in R_N : F_N(s,t)  > 0 \,\biggr\},
\]
\[
R_N^{--}
= \biggl\{\, (s,t)\in R_N : F_N(s,t) < 0 \,\biggr\},
\]
\[
R_N^{0}
= \biggl\{\, (s,t)\in R_N : F_N(s,t)  = 0 \,\biggr\}.
\]

It is straightforward to verify  that the disjoint sets $R_N^{++}$ and $R_N^{--}$ 
have equal measure; the problem is that it may not be the case that their union is equal to 
$R_N$.  To arrange for this to happen, we find disjoint subsets 
$R_{N}^{0,+}$ and $R_{N}^{0,-}$ of $R_N^{0}$ of equal measure whose union 
is $R_N^{0}$ and we define 
\[
R_N^{+} = R_N^{++} \cup R_{N}^{0,+}, 
\qquad 
R_N^{-} = R_N^{--} \cup R_{N}^{0,-}.
\]
Then we have $R_N^{+} \cup R_N^{-} = R_N$ and by construction 
$|R_N^{+}| = |R_N^{-}| = |R_N|/2$.  Moreover we have that 
\[
F_N(s,t)  \ge 0 \quad \text{on } R_N^{+},
\qquad 
F_N(s,t)  \le 0 \quad \text{on } R_N^{-}.
\]

We now decompose:
\begin{align*}
&\hspace{-0.3in}\int_{R_N}|F_N(s, t) |^2 \, ds\, dt\\
&=\int_{R_N^+\cap E}|F_N(s, t) |^2\, ds\, dt
 +\int_{R_N^-\cap E}|F_N(s, t) |^2\, ds\, dt
 +\int_{R_N\cap E^\complement}|F_N(s, t) |^2\, ds\, dt.
\end{align*}
Clearly, we have 
\[
\int_{R_N^-\cap E}|F_N(s, t)|^2\, ds\, dt \le \int_{R_N^-}|F_N(s, t)|^2\, ds\, dt = \frac12 \int_{R_N}|F_N(s, t)|^2\, ds\, dt,
\]
because $|R_N^-|=\frac12|R_N|$.

For the integral over $R_N\cap E^\complement$ we write $|R_N|=2^{-N}$ and rescale $t\mapsto u$.  Then
\[
\int_{R_N\cap E^\complement} |F_N(s, t)|^2 \, ds\, dt
\le 
\Bigg(\frac{|R_N\cap E^\complement|}{|R_N|} + \sqrt{\frac{|R_N\cap E^\complement|}{|R_N|}}\,\Bigg)
\int_{R_N}|F_N(s, t)|^2 \, ds\, dt,
\]
by applying \eqref{sqrt} to the rescaled series $\sum_{n,k\ge1} a_{n+N,k+N}r_n(s)r_k(t)$.

Putting things together yields 
\[
\Big(\tfrac12 - \tfrac{|R_N\cap E^\complement|}{|R_N|} - \sqrt{\tfrac{|R_N\cap E^\complement|}{|R_N|}}\Big)
\int_{R_N}|F_N(s, t)|^2 \, ds\, dt
\le 
\int_{R_N^+\cap E}|F_N(s, t)|^2 \, ds\, dt,
\]
or the analogous inequality with $R_N^+$ replaced by $R_N^-$.

Recall that
\[
F(s,t)=H_{00}(s,t) + F_N(s,t),\qquad
H_{00}(s,t)=\sum_{0\le n,k\le N} a_{n,k} r_n(s)r_k(t),
\]
where $H_{00}(s,t)=C$ is a constant, since $r_n$ for $n\le N$ are constant on $I_N$, and $r_k$ for $k\le N$ are constant on $J_N$.
Assume  first that $C>0$.  
Then on $R^+_{{N}}$, we have $F_N\ge0$, hence
\[
|F_N(s,t)| = F_N(s,t) \le C+F_N(s,t)=|F(s,t)|.
\]
Thus
\[
\int_{R_N^+\cap E}|F_N(s,t)|^2 \, ds\, dt\le \int_{R_N^+\cap E}|F(s,t)|^2 \, ds\, dt.
\]

For the alternative case $C<0$, we have that   $F_N\le 0$ on $R^-_{ N}$, hence
\[
|F_N(s,t)| = -F_N(s,t) \le -C-F_N(s,t)=|F(s,t)|.
\]
Thus
\[
\int_{R_N^-\cap E}|F_N(s,t)|^2 \, ds\, dt\le \int_{R_N^-\cap E}|F(s,t)|^2 \, ds\, dt.
\]

Then without loss of generality, we can assume that $C>0$, then by \eqref{boundF} and the choice of $\varepsilon$, we have
\begin{align*}
\sum_{\substack{n>N\\k>N}} |a_{n,k}|^2 
&\le \frac{1}{|R_N|}\int_{R_N}|F_N(s,t)|^2ds dt \\
& \le
\frac{1}{|R_N|}\frac{1}{\tfrac12-\varepsilon-\sqrt{\varepsilon}}
\int_{R_N^+\cap E}|F(s,t)|^2 ds dt\\
&\le \frac{2}{|R_N|}\int_{R_N^+\cap E}|F(s,t)|^2ds dt .
\end{align*}
Since $|R_N^+|=\frac12|R_N|$, this becomes
\[
|R_N|\sum_{\substack{n>N\\k>N}} |a_{n,k}|^2 \le \frac{1}{|R_N^+|}\int_{R_N^+\cap E}|F(s, t) |^2 \, ds\, dt.
\]

Thus we know that there exists $N=N(E)$ such that 
\[
\sum_{\substack{n>N\\k>N}} |a_{n,k}|^2
\le \frac{1}{|R_N|}\; \operatorname*{ess\,sup}_{(s,t)\in E} |F(s,t)|^2,
\]
Taking square roots proves the lemma with $A(E)=\sqrt{\frac{1}{|R_N|}} = 2^N $.
\end{proof}

\section{The proof of Theorem~\ref{thm1}}\label{proofThm1} 

We now have the tools required to prove Theorem~\ref{thm1}. In this section we provide the proof. 

\begin{proof}[Proof of Theorem~\ref{thm1}] 
We assume that conclusion \eqref{con} is false and we will reach a contradiction. Then for each $n \in \mathbb{N}$, there exist non-zero a.e. functions $f'_n \in L^p(M)$, $    h'_n \in L^q(M)$, and $\alpha_n > 0$ such that:
\[
\mu\left( \left\{ x \in M : T^*\left(\frac{f'_n}{\alpha_n}, \frac{h'_n}{\alpha_n} \right)(x) > 1 \right\} \right) \geq n \left\| \frac{f'_n}{\alpha_n} \right\|_{L^p}^r \left\| \frac{h'_n}{\alpha_n} \right\|_{L^q}^r.
\]
Thus we can reduce matters  to the following situation: $f_n=\frac{f'_n}{\alpha_n} \in L^p(M)$, $h_n=\frac{h'_n}{\alpha_n} \in L^q(M)$, $E_n := \{ x : T^*(f_n, h_n)(x) > 1 \}$. Then:
\[
\mu(E_n) \geq n \|f_n\|_{L^p}^r \|h_n\|_{L^q}^r.
\]
Apply Lemma~\ref{lm21} to extract a subsequence $n_k$ of the natural numbers
such that the triple $(f_{n_k}, h_{n_k}, E_{n_k})$ satisfies 
\begin{align}\begin{split}\label{999}
    \sum_k \|f_{n_k}\|_{L^p}^r \|h_{n_k}\|_{L^q}^r &< \infty, \\
    \sum_k \mu(E_{n_k}) &= \infty.
\end{split}\end{align}

By bilinearity, it is not hard to show that for any $\alpha>0$, condition \eqref{999} for $(f_{n_k}, h_{n_k}, E_{n_k})$ and for $(\alpha f_{n_k}, \alpha^{-1}h_{n_k}, E_{n_k})$ are equivalent. Now we take $\alpha=\alpha_{n_k}=||f_{n_k}||_{L^p}^{-\frac{r}{q}}||h_{n_k}||_{L^q}^\frac{r}{p}$, then we have that:
\[
||\alpha_{n_k}f_{n_k}||_{L^p}^p=\alpha_{n_k}^p||f_{n_k}||_{L^p}^p=||f_{n_k}||_{L^p}^{-\frac{rp}{q}}||h_{n_k}||_{L^q}^\frac{rp}{p}||f_{n_k}||_{L^p}^p= \|f_{n_k}\|_{L^p}^r \|h_{n_k}\|_{L^q}^r.\]
Similarly,  we have 
$$
||\alpha_{n_k}^{-1}g_{n_k}||_{L^q}^q=\|f_{n_k}\|_{L^p}^r \|h_{n_k}\|_{L^q}^r .
$$

For notational convenience, we replace $(\alpha_{n_k}f_{n_k}, \alpha_{n_k}^{-1}h_{n_k}, E_{n_k})$ by $(f_n, h_n, E_n)$. Then we have the following situation:

For $n\in\mathbb N$, there are $f_n \in L^p(M)$, $h_n \in L^q(M)$, and   $E_n = \{ x : T^*(f_n, h_n)(x) > 1 \}$ such that 
\[
\mu(E_n) \geq n \|f_n\|_{L^p}^r \|h_n\|_{L^q}^r,
\]
\[\sum_n \| f_n\|_{L^p}^p   =\sum_k \|f_{n_k}\|_{L^p}^r \|h_{n_k}\|_{L^q}^r <\infty,
\]
\[
\sum_n \| h_n\|_{L^q}^q  =\sum_k \|f_{n_k}\|_{L^p}^r \|h_{n_k}\|_{L^q}^r <\infty.
\]

Apply Lemma~\ref{lm31} to choose $R_n \to \infty$ such that:
\[
\sum_n \|R_n f_n\|_{L^p}^p  < \infty,
\]
\[
\sum_n \|R_n h_n\|_{L^q}^q  < \infty.
\]

For a  function $F$ on $M\times [0,1]$ we introduce the slices
$$
F_x(t)=F^t(x)=F(x,t),  \qquad H_x(s)=H^s(x)=H(x,s).
$$

We claim that there exists a pair of functions $F(x,t)$, $H(x,s)$ both defined on $M\times [0,1]$ 
with the   properties:
\begin{enumerate}[label=(\roman*)]
    \item For almost every \( (t,s) \in [0,1]^2 \), we have \( F^t \in L^p(M) \), \(H^s \in L^q(M) \), and hence by assumption we have  \( T_m(F^t, H^s)  \in L^r(M) \) for each \( m \).
    \item For almost every \( (t,s)\in [0,1]^2 \), the maximal function
    \[
    T^*(F^t, H^s)  (x):= \sup_{m \geq 1} |T_m(F^t, H^s) (x)| = \infty
    \]
    for almost every \( x \in M \).
\end{enumerate}

Note that the validity of (i) and (ii) leads to a contradiction since the assumption is that for every \( f \in L^p(M) \), \( h \in L^q(M) \), the limit
\[
\lim_{m \to \infty} T_m(f, h)(x)
\]
exists for almost every \( x \in M \). Thus, the desired inequality \eqref{con} must hold.

Let $I=[0,1]$.   
We define measurable functions  on $M \times I$ by setting
$$F_N(x,t)=\sum_{n=1}^N r_n(t)R_n\tau_{g_n}(f_n)(x),$$
$$H_N(x,s)=\sum_{k=1}^N r_k(s)R_k\tau_{g_k}(h_k)(x),$$
for $x\in M$. For $ N_2>N_1$ natural numbers we consider  
$$\int_I \int_M |F_{N_1}(x,t) - F_{N_2}(x,t)|^pdxdt.$$
We need the following observation. Let
$$F(t) = \sum_{n=N_1}^{N_2} b_nr_n(t),$$
then for $1 \le p \le 2$
$$\int_0^1|F(t)|^p dt\le\bigg(\int_0^1|F(t)|^2dt\bigg)^{\frac{p}{2}}= \bigg(\sum_{n=N_1}^{N_2} |b_n|^2\bigg)^{\frac{p}{2}} \le\sum_{n=N_1}^{N_2} |b_n|^p.$$
Notice $||f_n||_{L^p(M)} = ||\tau_{g_n}(f_n)||_{L^p(M)}$ by translation invariance, then by Fubini's theorem, we obtain
$$\int_I \int_M |F_{N_1}(x,t) - F_{N_2}(x,t)|^pdxdt \le \sum_{n=N_1}^{N_2} ||R_nf_n||_{L^p}^p,$$
and similarly 
$$\int_I \int_M |H_{N_1}(x,s) - H_{N_2}(x,s)|^qdxds \le \sum_{n=N_1}^{N_2} ||R_nh_n||_{L^q}^q.$$
Then the sequence $F_N(x,t)$ is Cauchy in $L^p(M \times I)$, since the series $\sum_{n=1}^\infty ||R_nf_n||_{L^p}^p$ is convergent. Hence there exists a function  $F(x,t) $ in $L^p(M \times I)$ such that 
$$\|F_N(\cdot,\cdot)-F(\cdot,\cdot)\|_{L^p(M \times I)} \to 0$$
as $N \to \infty$. Define
\[
G_N(t)
:= \|F_N(\cdot,t)-F(\cdot,t)\|_{L^{p}(M)}^{p}
= \int_I |F_N(x,t)-F(x,t)|^{p}\,dt .
\]
By Fubini's theorem we have 
\[
\|F_N - F\|_{L^{p}(M\times I)}^{p}
= \int_I G_N(x)\,dt.
\]
Since $F_N \to F$ in $L^{p}(M\times I)$, the right-hand side converges to $0$.
Thus $G_N \to 0$ in $L^{1}(I)$.

Convergence in $L^{1}(I)$ implies the existence of a subsequence
$G_{N_j}$ of $G_N$ such that $G_{N_j}(t)\to 0$ for $t\in I$, except for a subset of measure zero which we denote by $\mathcal A$.
Equivalently, we have 
\begin{equation}\label{FB}
\|F_{N_j}(\cdot,t)-F(\cdot,t)\|_{L^{p}(M)}
\longrightarrow 0
\qquad\text{for } t\in I\setminus\mathcal A, \qquad  \textup{(as $j \to \infty$).}
\end{equation}
Similarly, there exists an $H(x,s)\in L^q(M \times I)$ such  that 
$$
\|H_N(\cdot,\cdot)-H(\cdot,\cdot)\|_{L^p(M \times I)} \to 0,
$$ 
and there is a subsequence $N_{j_k}$ of $N_j$ and a subset $\mathcal B $ of $I$ of measure zero such that
\begin{equation}\label{HB}
||H_{N_{j_k}}(\cdot,s)-H(\cdot,s)||_{L^q(M)} \to 0\qquad\text{for } s\in I\setminus\mathcal B 
\qquad  \textup{(as $k \to \infty$).}
\end{equation}

For convenience, we drop the double index of the subsequence and we denote 
$F_{N_{j_k}}=F_N$ and $H_{N_{j_k}}=H_N$
We also recall the notation 
 $(F_{N})_x(t)=(F _{N})^t(x)=F_N(x,t)$ and $(H_{N})_x(t)=(H _{N})^t(x)=H_N(x,t)$

Then for any fixed $(s,t)\in (I\setminus\mathcal A)\times(I\setminus\mathcal B)$, since \eqref{FB} and \eqref{HB} are valid, we have that \( F^t=F(\cdot,t) \in L^p(M) \) and  \(H^s=H(\cdot,s) \in L^q(M) \).
This proves (i).

Let us now fix $m$. Since $T_m$ is bounded operator, in view of the preceding facts, we have for every fixed $t \in I\setminus\mathcal A $ and $s \in I\setminus\mathcal B $,
\begin{align*}
\big\|T_m&\big(F^t,H^s\big) -T_m\big( (F_{N_k^{t,s}})^t,(H_{N_k^{t,s}})^s\big) \big\|^{r}_{L^r(M)}\\
=&\big\|T_m\big(F^t,H^s-(H_{N_k^{t,s}})^s\big) +T_m\big(F^t -F_{N_k^{t,s}}^t,(H_{N_k^{t,s}})^s\big) \big\|^{r}_{L^r(M)}\\
\le&\big\|T_m\big(F^t,H^s-(H_{N_k^{t,s}})^s\big) \big\|^{r}_{L^r(M)}+\big\|T_m\big(F^t -(F_{N_k^{t,s}})^t,(H_{N_k^{t,s}})^s\big) \big\|^{r}_{L^r(M)}\\
\le&\big\|F^t\|_{L^p(M)}^{r}\|H^s-(H_{N_k^{t,s}})^s\big\|_{L^q(M)}^{r}+\big\|F^t-(F_{N_k^{t,s}})^t\|_{L^p(M)}^{r}\| (H_{N_k^{t,s}})^s\big\|_{L^q(M)}^{r} \to 0,
\end{align*}
as $k \to \infty$. 
Hence, for every $t \in I\setminus\mathcal A $ and $s \in I\setminus\mathcal B $ 
there is a subset $(\mathscr C_m)^{t, s}$ of $M$ 
of measure zero such that 
\begin{equation}\label{cc}
T_m\big(F^t,H^s\big)(x) =\sum_n \sum_k r_n(t)r_k(s)R_nR_kT_m\big(\tau_{g_n}(f_n),\tau_{g_k}(h_k)\big)(x),\quad x\in M\setminus (\mathscr C_m)^{t, s}.
\end{equation}
Denote $\mathscr C^{t, s}  =\bigcup_{m=1}^{\infty}(\mathscr C_m)^{t, s} $, which is still a subset of $M$ of measure zero. Then we have that for every $t \in I\setminus\mathcal A $ and $s \in I\setminus\mathcal B $, \eqref{cc} is true for all $m$ when  $x\in M\setminus (\mathscr C_m)^{t, s}$.

Moreover, by an argument similar to that above, there is a subset $\mathscr D$ of $M$ of measure zero, such that for every 
$x \in M\setminus \mathscr D$, there
is a subsequence $(N_j)_x$ of natural numbers such that $T_m\big(F_{(N_j)_x},H_{(N_j)_x}\big)(t,s)$ converges to $T_m\big(F_x,H_x\big)(t,s)$ in   $L^r(I^2)$. Then for every 
$x \in M\setminus \mathscr D$, there are subsets $(\mathscr E_m)_x$, $(\mathscr F_m)_x$ of $I$ 
of measure zero such that 
\begin{align}
\begin{split}\label{beq2}
 T_m\big(F_x,H_x\big)(t,s)  
 &=\sum_n \sum_k r_n(t)r_k(s)R_nR_kT_m\big(\tau_{g_n}(f_n),\tau_{g_k}(h_k)\big)(x),\\
 &\textup{whenever}\quad (t,s)\in I\setminus (\mathcal E_m)_x\times I\setminus(\mathcal F_m)_x.
\end{split}
\end{align}
Denote $\mathscr E_{x}  =\bigcup_{m=1}^{\infty}(\mathscr E_m)_{x} $ and $\mathscr F_x =\bigcup_{m=1}^{\infty}(\mathscr F_m)_x $ noting that  $\mathscr E_x$ and $\mathscr F_x$  are still subsets of $I$ of measure zero.

The next claim is to prove that there is a subset $\mathcal F$ of $M$ of measure zero  such that for all $x \in M\setminus \mathcal F$ we have
\begin{equation}\label{eq3}
\sum_n \sum_K R_n^2R_k^2| T_m\big(\tau_{g_n}(f_n),\tau_{g_k}(h_k)\big)(x)|^2<\infty 
\qquad \textup{for every $m=1,2,\dots$}. 
\end{equation}
This may be seen as follows: First, $p,q\ge 1$, and $\sum ||R_nf_n||^p_{L^p}< \infty$, $\sum ||R_kh_k||^q_{L^q}< \infty$ by our construction. Next, since $||f_n||_{L^p} = ||\tau_{g_n}(f_n)||_{L^p}$ and  $||h_k||_{L^q} = ||\tau_{g_n}(h_k)||_{L^q}$ and each $T_m$ is a bounded operator, then we have that
\begin{align*}
&\hspace{-0.3in}\sum_n\sum_k\int_M| R_n  R_k T_m(\tau_{g_n}(f_n),\tau_{g_k}(h_k))(x)|^rd\mu\\
=&\sum_n\sum_k R_n
^rR_k^r\int_M|T_m(\tau_{g_n}(f_n),\tau_{g_k}(h_k))(x)|^rd\mu\\
\le&\sum_n\sum_k R_n
^rR_k^r C_m^r \|f_n\|^r_{L^p(M)}\|h_k\|^r_{L^q(M)}\\
=& C_m^r\sum_n\sum_k \Big(\|R_nf_n\|_{L^p(M)}\|R_kh_k\|_{L^q(M)}\Big)^r\\
\le&C_m^r\bigg(\sum_n \|R_nf_n\|^p_{L^p(M)}\bigg)^{\frac{r}{p}}\cdot \bigg(\sum_k \|R_kh_k\|^q_{L^q(M)}\bigg)^{\frac{r}{q}},
\end{align*}
the last inequality is obtained by   H\"older's inequality. 
Thus,  there is a subset $\mathcal F$ of $M$ of measure zero  such that for all $x \in M\setminus \mathcal F$ one has
$$
\sum_n\sum_k| R_n  R_k T_m(\tau_{g_n}(f_n),\tau_{g_k}(h_k))(x)|^r<\infty
\qquad \textup{for every $m=1,2,\dots$}. 
$$
Since $\frac{1}{2} \le r \le 1$, \eqref{eq3} is proved.

Combining the above facts we conclude that  for every $x\in M\setminus \mathcal F$, the Rademacher series \eqref{beq2} is a well-defined  function of $(t,s)$ on the set  $(I\setminus \mathscr E_x)\times (I\setminus\mathscr F_x)$.

Suppose now, contrary to (ii), that $T^*(F^t,H^s)(x) < \infty$ for a set of positive
$(x,t,s)$ measure. Then there exists a set $S$ of positive measure in $M \times I^2$, where $T^*(F^t,H^s)(x)$ is bounded say by $A$. That is, $T^*(F^t,H^s)(x) \le A$, whenever $(x,t,s) \in S$.
And therefore, for every $m$,
\begin{equation}\label{beq4}
|T_m(F^t,H^s)(x) | \leq A ,\quad (x,t,s)\in S.
\end{equation}
Let now $E_x=S \cap \big(\{
x\}\times I^2\big)$, $E_x$ may be considered for each $x \in M\setminus \mathcal F$, a subset
of $I^2$ (more specifically, a subset of $(I\setminus \mathscr E_x)\times ( I\setminus\mathscr F_x)$). Since $S$ has positive measure in $M \times I$, $E_x$ is Lebesgue measurable
for almost every $x$, and there is a subset $M_0$ of $M$ with positive measure, such that $|E_x|>0$ for all $x \in M_0$.

So now we can apply Lemma~\ref{lm33}, to the case of the Rademacher series \eqref{beq2}, remembering \eqref{beq4};  the set $  E$ in Lemma~\ref{lm33} will be $E_x$, where $x \in M_0$. Then there exists a $N(x),A(E_x)$ independent of $s,t,m$ such that 
\begin{align*}
& \hspace{-0.3in}\bigg(\sum_{\substack{n \ge N(x) \\  k \ge N(x)}} \bigg|r_n(t)r_k(s)R_nR_kT_m\big(\tau_{g_n}(f_n),\tau_{g_k}(h_k)\big)(x)\bigg|^2\bigg)^{\frac{1}{2}} \\
&\le\
A(E_x)\,\operatorname*{ess\,sup}_{(s,t)\in E}|T_m(F_x,H_x)(s,t)|\\
&\le A(E_x)\cdot A=:A(x), 
\end{align*}
thus we have the following:
$$\bigg(\sum_{\substack{n \ge N(x) \\  k \ge N(x)}} | R_nR_k T_m\big(\tau_{g_n}(f_n),\tau_{g_n}(h_n)\big)(x)|^2\bigg)^{\frac{1}{2}} \le A(x),\quad x\in M_0,$$
where  $A(x)$, $N(x)$ are independent of $m$.
Hence 
\begin{equation}\label{beq5}
| R_n^2T_m\big(\tau_{g_n}(f_n),\tau_{g_n}(h_n)\big)(x)|\le A(x),\quad x\in M_0,
\end{equation}
if $n > N(x)$.
Now $T_m(\tau_{g_n}f_n, \tau_{g_n}h_n) = \tau_{g_n} T_m(f_n,h_n)$. Taking the $\sup$ (over $m$) of the left side of \eqref{beq5},
$$
R_n^2(T^*(f_n,h_n))(g_n^{-1}(x)) \le A(x),\quad x \in M_0,
$$
whenever $n > N(x)$.

Now if $x_0 \in F_0$, by Lemma~\ref{lm24}, then $x_0$ is contained in infinitely many $F_n$, $F_n = g_n[E_n]$. But $T^*(f_n,h_n)(y_n) > 1$, if $y_n \in E_n$. Therefore $x_0 \in F_0$, implies
$R_n^2T^*(f_n,h_n)(g_n^{-1}(x_0)) > R_n^2  \to \infty$, for infinitely many $n$'s. Hence if $x \in F_0$, these
are infinitely many $n$ so that:
$$R_n^2T^*(f_n,h_n)(g_n^{-1}(x_0)) > A(x_0).$$
Thus \eqref{beq5} implies that if $x_0\in M_0$, then $x_0\not\in F_0$. This shows that $M_0$ is the subset of the complement of $F_0$, and therefore
$\mu(M_0) = 0$, contrary to what was found earlier.
We have therefore proved that the function $F(x, t)$ and $H(x,s)$ satisfies the two
conditions (i) and (ii) above which leads to a contradiction with the hypothesis of Theorem~\ref*{thm1}, unless an inequality of the
type \eqref{con} holds for some $C$. This contradiction shows that failure of the weak-type bound would force divergence on a set of positive measure, completing the proof. 
\end{proof}

\section{The Extension to Borel measures}

In this section, we focus on  the case where $p=q=1$ and $r=1/2$ in which we extend Theorem~\ref{thm1} to Borel measures. This extension could be helpful in certain  applications; see Section \ref{APP}.   

Throughout  this section we let $(M,d\mu)$ be a compact abelian group equipped with Haar measure $\mu$. We denote by 
   $\mathcal C(M)$   the class of continuous functions on $M$  with the sup topology. 
For each $m=1,\dots,N$, let 
\[
T_m:L^1(M)\times L^1(M) \to \mathcal  C(M)
\]
be a bounded bilinear operator that admits the kernel representation
\[
T_m(f,h)(x)=\iint_{M\times M} K_m(x-y,\,x-z)\,f(y)\,h(z)\,d\mu(y)\,d\mu(z),
\quad K_m\in L^\infty(M\times M),
\]
such representations arise naturally for translation-invariant bilinear operators on compact groups.
Extend $T_m$ to finite Borel measures $(\nu_1,\nu_2)$ by the same formula and we denote this extension by  $ T_m(\nu_1,\nu_2) $. Let $\|\nu_1\|,\|\nu_2\|$ be the total variation of $\nu_1,\nu_2$ respectively. Then we have the following theorem.

\begin{theorem}\label{thm3}
Suppose that for every \( f, g \in L^1(M) \), we have
    \[
    \limsup_{m \to \infty} |T_m(f, g)(x)|<\infty
    \]
 on a set of positive measure (which may depend on \( f, g \)).

Then there exists a constant \( A > 0 \)   such that for all bounded Borel measures \( \nu_1, \nu_2 \) on \( M \), and all \( \alpha > 0 \),
\[
\mu\left( \left\{ x \in M : T^*(\nu_1, \nu_2)(x) > \alpha \right\} \right)
\leq \frac{A}{\sqrt\alpha} \|\nu_1\| \cdot \|\nu_2\| .
\]
\end{theorem}

We begin with a useful lemma.
\begin{lemma}\label{644}
Let $T_1,\dots,T_N$  be a finite collection of operators satisfying the conditions above. Let $\nu_1,\nu_2$ be finite Borel measures with finite total variation $\|\nu_1\|,\|\nu_2\|$ and 
set $h_m(x):=T_m(\nu_1,\nu_2)(x)$. Then there exist   sequences $\{f_k\}$, $\{g_k\}$ where $f_k,g_k\in L^1(M)$ with $\|f_k\|_{L^1}\le\|\nu_1\|$ and $\|g_k\|_{L^1}\le\|\nu_2\|$ such that for all $m=1,\dots,N$ we have
$$
\lim_{k\to \infty} T_m(f_k,g_k)(x) \to h_m(x)
$$
for almost all $x\in M$.
\end{lemma}

\begin{proof}
Let $\{\phi_k\}_{k\ge1}\subset \mathcal C(M)$ be a nonnegative approximate identity
with $\int_M\phi_k\,d\mu=1$, and define
\[
f_k:=\phi_k*\nu_1,\qquad g_k:=\phi_k*\nu_2 .
\]

Then \( f_k, g_k \in L^1(M) \), with
\[
\|f_k\|_{L^1} \leq \|\nu_1\|, \quad \|g_k\|_{L^1} \leq \|\nu_2\|.
\]
By Fubini's theorem  and the definitions of $f_k,g_k$ we have 
\[
\begin{aligned}
T_m(f_k,g_k)(x)
&=\iint K_m(x-y,x-z)\Big(\int \phi_k(y-u)\,d\nu_1(u)\Big)
                      \Big(\int \phi_k(z-v)\,d\nu_2(v)\Big)\,d\mu(y)\,d\mu(z) \\
&=\iint\!\!\Big(\iint \phi_k(r)\,\phi_k(s)\,K_m\big(x-(u+r),\,x-(v+s)\big)\,d\mu(r)\,d\mu(s)\Big)
    d\nu_1(u)\,d\nu_2(v),
\end{aligned}
\]
Then, writing
\[
K_m^{(k)}(a,b):
=\iint \phi_k(r)\,\phi_k(s)\,K_m(a-r,b-s)\,d\mu(r)\,d\mu(s),
\]
we have for each $m$ and all $x\in M$,
\[
T_m(f_k,g_k)(x)
=\iint_{M\times M} K_m^{(k)}(x-u,\,x-v)\,d\nu_1(u)\,d\nu_2(v).
\]

Since $\phi_k\ge 0$ and $\int_M \phi_k\,d\mu=1$, we have for all $(a,b)\in M\times M$,
\[
\begin{aligned}
\big|K_m^{(k)}(a,b)\big|
&\le \iint \phi_k(r)\,\phi_k(s)\,\big|K_m(a-r,b-s)\big|\,d\mu(r)\,d\mu(s) \\
&\le \|K_m\|_{L^\infty(M\times M)} \iint \phi_k(r)\,\phi_k(s)\,d\mu(r)\,d\mu(s) \\
&= \|K_m\|_{L^\infty(M\times M)}.
\end{aligned}
\]
Hence $|K_m^{(k)}|\le \|K_m\|_{L^\infty}$ pointwise for all $k$. 

By the triangle inequality and  Fubini's theorem we obtain,
\[
\begin{aligned}
&\hspace{-0.3in}\|T_m(f_k,g_k)-h_m\|_{L^1(M)}\\
&= \int_M \left| \iint_{M\times M} (K_m^{(k)}-K_m)(x-u,x-v)\, d\nu_1(u)\, d\nu_2(v) \right| d\mu(x) \\
&\le \iint_{M\times M} \left( \int_M \big|(K_m^{(k)}-K_m)(x-u,x-v)\big|\, d\mu(x) \right) d|\nu_1|(u)\, d|\nu_2|(v)\\
&\le  2\|K_m\|_{L^\infty}\iint_{M\times M} d|\nu_1|(u)\, d|\nu_2|(v)\\
&=2\|K_m\|_{L^\infty}\cdot\|\nu_1\|\cdot\|\nu_2\|
\end{aligned}
\]

With the bound above, if we can show that there is a subsequence $\{k_i\}_i$ such that $\lim_{i \to \infty} K_m^{(k_i)}(a,b)=K_m(a,b)$ a.e., then by dominated convergence theorem, we have that for subsequences $f_{k_i}$ and $g_{k_i}$, 
$$
\lim_{ i \to \infty} T_m(f_{k_i},g_{k_i})(x) \to h_m(x) \qquad\textup{a.e.}
$$ 

Since $M$ is compact and $K_m\in L^\infty(M\times M)$, we have $K_m\in L^1(M\times M)$,
and because $\phi_k\otimes\phi_k$ is an approximate identity on $M\times M$,
$\|K_m^{(k)}-K_m\|_{L^1(M\times M)} \to0$.
Therefore, there exists a subsequence $\{k_i\}_i$ such that $\lim_{i \to \infty} K_m^{k_i}(a,b)=K_m(a,b)$ a.e. for all $m=1,\dots,N$. This completes the proof.
\end{proof}

\begin{proof}[Proof of Theorem~\ref{thm3}]
We begin by noticing that
\begin{itemize}
    \item Each \( T_m \) is bounded: there exists \( C_m> 0 \) such that
    \[
    \| T_m(f, g) \|_{\infty} \leq C_m \|f\|_{L^1(M)} \|g\|_{L^1(M)}.
    \]
    
    \item Each \( T_m \) commutes with translations. 
\end{itemize}

Since \( T_m \) is bilinear, bounded, and translation invariant, then by   Theorem~\ref*{thm1} there is a constant $A<\infty$
such that 
\begin{equation}\label{994}
\mu\left( \left\{ x : T^*(f_k, g_k)(x) > \alpha \right\} \right)
\leq \frac{A}{\sqrt\alpha} \|f_k\|_{L^1} \cdot \|g_k\|_{L^1}
\leq \frac{A}{\sqrt\alpha} \|\nu_1\| \cdot \|\nu_2\|.
\end{equation}
Let 
$$
T_N^*(f,h) =\sup_{1\le m\le N}|T_m(f,h)|.
$$

By Lemma~\ref{644}, there are sequences  \( f_k, g_k \in L^1(M) \)   such that :
\[
\lim_{k \to \infty}T_m(f_k, g_k)(x)  =T_m(\nu_1, \nu_2)(x),
\]
 for each \(m=1,2,\dots, N \) and for almost all  \(   x \in M \).

Define $E^\alpha:=\{ x \in M : T^*(\nu_1, \nu_2)(x) > \alpha \}$ and $E_N^\alpha:=\{ x \in M : T_N^*(\nu_1, \nu_2)(x) > \alpha \}$, notice that the sets $E_N^\alpha$ are increasing with $N$ and that the following are valid
\[ 
\mu( E_N^\alpha )=
\mu\left( \left\{ x : T_N^*(\nu_1,\nu_2)(x) > \alpha \right\} \right)\le \mu\left( \left\{ x : T^*(\nu_1,\nu_2)(x) > \alpha \right\} \right)=\mu( E^\alpha ).
\]

Next, we will prove that 
\[
\limsup_{k  \to \infty}
\mu\left( \left\{ x : T_N^*(f_k, g_k)(x) > \alpha \right\} \right)
\geq \mu\left( \left\{ x :  T^*(\nu_1, \nu_2)(x) > \alpha \right\} \right).
\]

For $x\in M$, by Lemma~\ref{644}, we have that $\lim_{k \to\infty}|T_m(f_k,g_k)(x)|= |T_m(\nu_1,\nu_2)(x)|$ for each fixed 
$m\in\{1,2,\dots, N\}$ and for almost all $x\in M$.
Hence, for a.e.\ $x$,
\[
\liminf_{k \to\infty}\  T_N^*(f_k,g_k)(x)\ \ge\ \sup_{1\le m\le N}\ \liminf_{k \to\infty} |T_m(f_k,g_k)(x)|
=T_N^*(\nu_1,\nu_2)(x).
\]
Therefore,
\[
\mu\big( \{x:\,T_N^*(\nu_1,\nu_2)(x)>\alpha\}\big) \le \mu\big(\{x:\,\liminf_{k} T_N^*(f_k,g_k)(x)>\alpha\}\big).
\]
By Fatou's lemma, we have that 
\[
\mu \big( \{x:\,\liminf_{k} T_N^*(f_k,g_k)(x)>\alpha\}\big)\le \liminf_{k \to\infty}\mu \big(\{x:\, T_N^*(f_k,g_k)(x)>\alpha\}\big).
\]
Thus, we have that 
$$
\mu(E_N^\alpha)\le\liminf_{k \to\infty}\mu\big( \{x:\, T_N^*(f_k,g_k)(x)>\alpha\}\big)\le\liminf_{k \to\infty}\mu\big( \{x:\, T^*(f_k,g_k)(x)>\alpha\}\big).
$$

Combining with the bound in \eqref{994}, we deduce  
$$
\mu(E_N^\alpha)\le \frac{A}{\sqrt\alpha} \|\nu_1\| \cdot \|\nu_2\|. 
$$
In view of  the monotonicity of $E_N^\alpha$, letting $N \to \infty$, we obtain  
\[
\mu\left( \left\{ x : T^*(\nu_1,\nu_2)(x) > \alpha \right\} \right)
\leq \frac{A}{\sqrt\alpha} \|\nu_1\| \cdot \|\nu_2\| 
\]
and this concludes the proof. 
\end{proof}

\section{The Extension to Positive Operators}\label{Sawyer}
Sawyer~\cite{Sawyer1966} removed the restriction $p\le 2$ in Stein's theorem \cite{Stein1961} under the  assumption 
of positivity on the linear operators. Theorem~\ref{thm2}  below extends  Sawyer's result   to positive bilinear operators     which commute  with an ergodic family of measure-preserving transformations; precisely, we prove that for such operators a.e. finiteness implies a weak-type estimate   from  $L^p\times L^q$ to weak $L^r$ in the   range $1\le p,q<\infty$ where $1/r=1/p+1/q$. The proof uses a delicate inductive construction, analogous to that in ~\cite{Sawyer1966}, and relates  to Stein's probabilistic argument described in 
the proof of Theorem~\ref{thm1}.

Let $(X,\mathcal{B},\mu)$  be a finite measure space normalized so that  $\mu(X)=1$.

\begin{definition}[continuous-in-measure]
Let
\[
\mathscr B_m : L^p(X)\times L^q(X) \to \{\text{measurable functions on } X\}
\]
be a sequence of bilinear operators.  
We say that each $\mathscr B_m$ is \emph{continuous-in-measure} if the following holds:
whenever
\[
f_k  \to f \quad \text{in } L^p(X),
\qquad
g_k  \to g \quad \text{in } L^q(X),
\]
we have, for every fixed $m$,
\[
\mathscr B_m(f_k,g_k)(x)\longrightarrow \mathscr B_m(f,g)(x)
\quad\text{in measure.}
\]
Equivalently, for every $\varepsilon>0$,
\[
\mu\Bigl(\, x:\,
\bigl|\mathscr B_m(f_k,g_k)(x)-\mathscr B_m(f,g)(x)\bigr|>\varepsilon
\,\Bigr)
\longrightarrow 0,\quad \text{as } k \to \infty.
\]

\end{definition}

\begin{definition}[Positive operator]
A    bilinear operator $\mathscr B$ defined on $L^p(X)\times L^q(X)$ and taking values in the set of 
measurable functions of another measure space 
is called \emph{positive} if
$f\ge 0$ and $h\ge 0$ a.e. imply
$\mathscr B(f,h)\ge 0$ a.e.. 
Precisely, for any $f$ in $L^p(X)$ and $h$ in $L^q(X)$ there is a 
set of measure zero $\mathcal E^{\mathscr B(f,h)}$ such that  $\mathscr B(f,h)(x)\ge 0$ for  
$x\in X\setminus \mathcal E^{\mathscr B(f,h)}$.
\end{definition}

\begin{lemma}\label{abslo}
Let $\mathscr B$ be a positive bilinear operator, meaning that
\[
\mathscr B(f,h)(x)\ge 0
\quad\text{for almost every }x\in X
\quad\text{whenever } f,h\ge 0.
\]

Then the following properties hold.

\medskip

\noindent\textbf{\textup{(i) Monotonicity.}}
If $f_1,f_2\in L^p(X)$ and $h_1,h_2\in L^q(X)$ satisfy
\[
0\le f_1 \le f_2,
\qquad
0\le h_1 \le h_2
\quad\text{a.e. on }X,
\]
then
\[
\mathscr B(f_1,h_1)(x)
\;\le\;
\mathscr B(f_2,h_2)(x)
\quad\text{for almost every }x\in X,
\]
more precisely, for all
\[
x\in X \setminus
\Big(
\mathcal{E}^{\mathscr B(f_2-f_1,h_1)}
\cup
\mathcal{E}^{\mathscr B(f_1,h_2-h_1)}
\cup
\mathcal{E}^{\mathscr B(f_2-f_1,h_2-h_1)}
\Big).
\]

\medskip

\noindent\textbf{\textup{(ii) Domination by absolute values.}}
For all $f\in L^p(X)$ and $h\in L^q(X)$,
\[
|\mathscr B(f,h)(x)|
\;\le\;
\mathscr B(|f|,|h|)(x)
\quad\text{for almost every }x\in X,
\]
more precisely, for all
\[
x\in X \setminus
\bigcup_{\varepsilon,\delta\in\{+,-\}}
\mathcal{E}^{\mathscr B(f^\varepsilon,h^\delta)},
\]
where $f = f^+ - f^-$ and $h = h^+ - h^-$  and $f^+,f^-,h^+,h^-\ge 0$.
\end{lemma}

\begin{proof}

\noindent\textbf{Proof of (i).}
Since $f_1 \le f_2$ and $h_1 \le h_2$, write
\[
f_2 = f_1 + u,
\qquad
h_2 = h_1 + v,
\qquad
u,v \ge 0.
\]
By bilinearity we have 
\[
\mathscr B(f_2,h_2)
=
\mathscr B(f_1+u,h_1+v)
=
\mathscr B(f_1,h_1)
+
\mathscr B(u,h_1)
+
\mathscr B(f_1,v)
+
\mathscr B(u,v).
\]
By positivity, each of the additional terms is non-negative a.e..
Therefore,
\[
\mathscr B(f_2,h_2)(x)
\;\ge\;
\mathscr B(f_1,h_1)(x)
\quad
\text{for all }
x \notin
\mathcal{E}^{\mathscr B(f_2-f_1,h_1)}
\cup
\mathcal{E}^{\mathscr B(f_1,h_2-h_1)}
\cup
\mathcal{E}^{\mathscr B(f_2-f_1,h_2-h_1)}.
\]
This proves (i).

\medskip

\noindent\textbf{Proof of (ii).}
Write the sign decompositions
\[
f = f^+ - f^-,
\qquad
h = h^+ - h^-.
\]
By bilinearity,
\[
\mathscr B(f,h)
=
\mathscr B(f^+,h^+)
-
\mathscr B(f^+,h^-)
-
\mathscr B(f^-,h^+)
+
\mathscr B(f^-,h^-).
\]
Set
\[
A := \mathscr B(f^+,h^+) + \mathscr B(f^-,h^-),
\qquad
B := \mathscr B(f^+,h^-) + \mathscr B(f^-,h^+).
\]
By positivity,
\[
A(x)\ge 0,
\qquad
B(x)\ge 0,
\]
for all
\[
x \notin
\bigcup_{\varepsilon,\delta\in\{+,-\}}
\mathcal{E}^{\mathscr B(f^\varepsilon,h^\delta)}.
\]
Then
\[
\mathscr B(f,h) = A - B,
\qquad
|\mathscr B(f,h)| \le A + B.
\]
Finally,
\begin{align*}
A + B
&=
\mathscr B(f^+,h^+)
+ \mathscr B(f^+,h^-)
+ \mathscr B(f^-,h^+)
+ \mathscr B(f^-,h^-) \\
&=
\mathscr B(f^+ + f^-,\, h^+ + h^-) \\
&=
\mathscr B(|f|,|h|),
\end{align*}
since $|f| = f^+ + f^-$ and $|h| = h^+ + h^-$.

Thus
\[
|\mathscr B(f,h)(x)| \le \mathscr B(|f|,|h|)(x)
\]
outside the stated exceptional set. This completes the proof.
\end{proof}

\begin{definition}[Measure-preserving transformation]
A measurable map $\tau:X \to X$ is called \emph{measure-preserving} if
\[
\mu\bigl(\tau^{-1}[A]\bigr)=\mu(A)\qquad\text{for all }A\in\mathcal B.
\]
Equivalently, $\tau $ is  measure-preserving if for every integrable $h:X \to\mathbb C$ one has 
\[
\int_X h\circ \tau \,d\mu =\int_X h \,d\mu .
\]
\end{definition}

\begin{definition}[Ergodic family]
{\rm Let  \(\mathcal{T}\) be a collection of measure-preserving 
transformations on  a measure space \((X,\mathcal{B},\mu)\).
We say that \(\mathcal{T}\) is an \emph{ergodic family} on \(X\) 
if whenever a measurable set \(A\in\mathcal{B}\) satisfies
\[
   \tau^{-1}[A] = A
   \qquad\text{for every } \tau\in\mathcal{T},
\]
then 
\[
   \mu(A)=0 \qquad \text{or} \qquad \mu(A)=1.
\]}
\end{definition}

\begin{definition}[Bilinear distributive sequence]\label{Def8888}
{\rm Let
\[
\mathscr B_m : L^p(X)\times L^q(X) \to \{\text{measurable functions on } X\},
\qquad m\ge1,
\]
be a sequence of bilinear operators.  Following Sawyer~\cite{Sawyer1966},  we say that $\{\mathscr B_m\}_m$ is \emph{distributive on $X$} if there exists
an ergodic family $\mathcal T$ of measure-preserving transformations on $X$
such that, for every $\tau\in\mathcal T$ and every $f\in L^p(X)$, $h\in L^q(X)$,
the   relationship holds  for each $m$
\[
\mathscr B_m(f,h)\circ\tau( x)
\;=\;
\mathscr B_m(f\circ\tau,\; h\circ\tau)(x)
\qquad\text{for almost all } x\in X. 
\]
}
\end{definition}

One can show that if $\{\mathscr B_m\}_m$ is  distributive, then 
the maximal operator  
\[
\mathscr B^{*}(f,h) 
  := \sup_{m\ge1}\, \bigl|\mathscr B_m(f,h) \bigr|.
\]
is also distributive, i.e., it satisfies
\[
\mathscr B^*(f,h)\circ\tau( x)
\;=\;
\mathscr B^*(f\circ\tau,\; h\circ\tau)(x)
\qquad\text{for almost all }  x\in X. 
\]
 Precisely, the preceding holds for $X\setminus\mathcal C^{\mathscr B^{*}(f,h)\circ\tau}$, where $\mathcal C^{\mathscr B^{*}(f,h)\circ\tau}$ has measure zero.

After introducing all the necessary definitions, we are now ready to state the theorem.
\begin{theorem}\label{thm2}
Let $1 \le p,q \leq \infty$ and $\frac{1}{2} \leq r \leq \infty$, such that $\frac{1}{r} = \frac{1}{p} + \frac{1}{q}$. Let $\mathscr B_m$ be  a distributive sequence with the ergodic family  $\mathcal{T}$ of continuous-in-measure
positive bilinear transformations of $ L^p(X) \times L^q(X)$  to measurable functions on
$X$. Suppose that  for every $f \in L^p(X)$ and $h \in L^q(X)$ we have
\begin{equation}\label{kkdjkjd}
\mathscr B^*(f, h)(x) <\infty
\end{equation}
  for almost every $x \in X$.

Then, there exists a constant $C > 0$ such that for all $f \in L^p(X)$, $h \in L^q(X)$, and all $\alpha > 0$,
\begin{equation}\label{pqr44}
\mu\left(\left\{ x \in X : \mathscr B^*(f, h)(x) > \alpha \right\} \right) \leq \frac{C}{\alpha^r} \|f\|_{L^p(X)}^r \|h\|_{L^q(X)}^r.
\end{equation}
 In particular, if $p=q=\infty$, there exists $C_\infty$ such that
  \[
     \|\mathscr B^*(f,h)\|_{L^\infty(X)}
     \le C_\infty\,\|f\|_{L^\infty(X)}\|h\|_{L^\infty(X)}.
  \]
\end{theorem}

To prove Theorem~\ref{thm2} we need the following two lemmas.  

\begin{lemma}[Sawyer~\cite{Sawyer1966}]\label{mixing}
Let $\mathcal T$ be an ergodic family of measure-preserving
transformations on a probability space $(X,\mathcal{B},\mu)$.
Then for any two measurable sets $A,B\in\mathcal{B}$, and for any constant $\theta>1$, 
there exists a transformation $\tau\in\mathcal{T}$ such that
\[
   \mu\!\bigl( B \cap \tau^{-1}[A]\bigr)
   \;\le\;
   \theta\, \mu(A)\, \mu(B).
\]
\end{lemma}

\begin{lemma}\label{lm242}
Assume $\mathcal{T}$ is an ergodic family. Then, if $\{A_n\}$ is a sequence
of measurable subsets of $X$ such that
\[
\sum_{n=1}^\infty \mu(A_n) = \infty,
\]
there exists a sequence of transformations $\{\tau_n\} \subseteq \mathcal{T}$
such that
\begin{equation}\label{eq:22}
\mu\Big(\bigcup_{n \ge M} \tau_n^{-1}[A_n]\Big) = 1,
\qquad \text{for all } M.
\end{equation}
That is, $\tau_n(x) \in A_n$ infinitely often for almost every $x \in X$.
\end{lemma}
\begin{proof}
Fix $M$. Let $\{A_n\} \subseteq \mathcal{B}$ be any sequence of sets. Then, by Lemma~\ref{mixing}
and induction, we can choose transformations $\{\tau_n\} \subseteq \mathcal{T}$
such that for all $N>M$,
\begin{equation}\label{eqs1}
\mu\left(\tau_1^{-1}\big[A_1^\complement\big] \cap \tau_2^{-1}\big[A_2^\complement\big] \cap \cdots \cap \tau_N^{-1}\big[A_N^\complement\big]\right)
\le \theta_1 \theta_2 \cdots \theta_N \prod_{k=1}^N \mu(A_k^\complement),
\end{equation}
where $\{\theta_k\}$ is any sequence of constants with $\theta_k > 1$ and $\prod_{k=1}^\infty\theta_k<\infty$.
As discussed in Lemma~\ref{lm24}, since $\sum_{n=1}^\infty \mu(A_n) = \infty$ we must have $\sum_{n=M}^\infty \mu(A_n) = \infty$, thus by Lemma~\ref{ln}, we have
\begin{equation}\label{eqs0}
\prod_{k=M}^\infty \mu\big(A_k^\complement\big) = 0.
\end{equation}
Choose a sequence $\{\theta_k\}_{k\ge1}$ with $\theta_k>1$ such that
\[
\prod_{k=1}^\infty \theta_k < \infty .
\]
Let $\Theta := \prod_{k=1}^\infty \theta_k$, combining  with \eqref{eqs1}, we have
\[
\mu\Big( \bigcap_{k=M}^N \tau_k^{-1}\big[A_k^\complement\big]  \Big)
\le
\Big( \prod_{k=M}^N \theta_k \Big)
\prod_{k=M}^N \mu\big(A_k^\complement\big)
\le
\Theta \prod_{k=M}^N \mu\big(A_k^\complement\big).
\]
Letting $N \to \infty$ and using \eqref{eqs0}, we obtain
\[
\mu\Big( \bigcap_{k=M}^\infty \tau_k^{-1}\big[A_k^\complement\big]  \Big)=0.
\]
By De Morgan's law,
\[
\left(\bigcap_{k=M}^\infty \tau_k^{-1}\big[A_k^\complement\big]
\right)^\complement
= \bigcup_{k=M}^\infty \Big( \tau_k^{-1}\big(A_k^\complement\big) \Big)^\complement .
\]

For each $k$, since preimages commute with complements,
\[
\Big( \tau_k^{-1}\big(A_k^\complement\big) \Big)^\complement
= \tau_k^{-1}\big( (A_k^\complement)^\complement \big)
= \tau_k^{-1}(A_k).
\]
Therefore,
\[
E^\complement = \bigcup_{k=M}^\infty \tau_k^{-1}(A_k).
\]
Since $\mu(X)=1$, we have
\[
m\Big( \bigcup_{k=M}^\infty A_k^{w_k} \Big)=1.
\]
\end{proof}

\begin{proof}[Proof of Theorem~\ref{thm2}]
Using similar argument as in the proof of  Theorem~\ref{thm1}, we can find two sequences of functions
 $\{f_n\}_{n\in\mathbb N} \in L^p(X)$, $\{h_n\}_{n\in\mathbb N} \in L^q(X)$, and a sequence of real numbers $R_n \to \infty$ such that 
\[
\sum_n\mu( \{ x\in X : \mathscr B^*(f_n, h_n)(x) > 1 \}) =\infty,
\]
with $\sum_n \|R_n f_n\|_{L^p}^p<\infty$, and $\sum_n \|R_n h_n\|_{L^q}^q  <\infty$.
Thus we can take $f''_n(x)=R_nf_n(x)$ and $h''_n(x)=R_nh_n(x)$ and then we have  
\[
\sum_{j=1}^\infty
\mu\bigl(\{x\in X :\mathscr B^*(f''_n, h''_n)(x) > R_n^2\}\bigr)
\;=\;\infty,
\]
with 
$$
\sum_n \|f''_n\|_{L^p}^p<\infty, \qquad \textup{and}\qquad \sum_n \|h''_n\|_{L^q}^q  <\infty .
$$
For simplicity of notation, we replace $f''_n$ with $f_n$, $h''_n$ with $h_n$.

Since each $\mathscr B_m$ is positive, by Lemma~\ref{abslo} we have, for any $f\in L^p(X)$, $h\in L^q(X)$
\[
\mathscr B_m (f,h)(x) \;\le\; \mathscr B_m\big(|f|,|h|\big)(x)
 \quad\text{for  every }x\in X\setminus \mathcal H_m^{f,h},
\]
where $\mathcal H_m^{f,h}$ is a null subset of $X$. Hence, for $x\in X\setminus \bigcup_m \mathcal H_m^{f,h}$, we have
\[
\mathscr B^*(f,h)(x)
=
\sup_m |\mathscr B_m (f,h)(x)|
\le
\sup_m  \mathscr B_m\big(|f|,|h|\big)(x) 
= 
\mathscr B^*(|f|,|h|)(x).
\]
Therefore, replacing   $f_n$ by $|f_n|$, and $h_n$ by $|h_n|$ we can
assume, without loss of generality, that $f_n ,h_n \ge 0  $ for all $n=1,2,\dots$ 
and the level sets
\[
A_n := \{ x\in X : \mathscr B^*(f_n, h_n)(x) > R_n^2 \}
\]
still satisfy
\begin{equation}\label{eq:Aj-sum}
\sum_{n=1}^\infty \mu(A_n) = \infty,\quad\sum_n \|f_n\|_{L^p}^p<\infty,\quad \sum_n \|h_n\|_{L^q}^q  <\infty.
\end{equation}

Now apply Lemma~\ref{lm242} to the sequence of sets $\{A_n\}$
in \eqref{eq:Aj-sum}. Since $\sum_n \mu (A_n)=\infty$, Lemma~\ref{lm242} ensures
that there exists a sequence $\{\tau_n\}\subset \mathcal{T}$ and a null subset $\mathcal N_0$ of $X$ such that
for   every $x\in X\setminus \mathcal N_0$ we have
\[
\tau_n(x) \in A_n\quad\text{for infinitely many }n.
\]
Equivalently, for   every $x\in X\setminus \mathcal N_0$
\[
\mathscr B^*(f_n,h_n)\circ\tau_n(x) > R_n^2
\quad\text{for infinitely many } n  ,
\]
and hence
\begin{equation}\label{eq:sup-infty-on-orbit}
\sup_{n\ge1} \mathscr B^*(f_n,h_n)\circ\tau_n(x)= \infty
\quad\text{for all   }x\in X\setminus \mathcal N_0.
\end{equation}

Notice that since $\tau_n$ is measure-preserving, the $L^p$ norms
  are preserved, that is
  \[
  \|f_n\circ\tau_n\|_{L^p}^p=
  \|f_n\|_{L^p}^p,\quad\|h_n\circ\tau_n\|_{L^q}^q=
  \|h_n\|_{L^q}^q.
  \]

Moreover, by the assumption of commutation relation, 
we have for each $n$ and $\forall  f \in L^p(X)$, $\forall h \in L^q(X)$,
\[
\mathscr B^*(f\circ\tau_n, h\circ\tau_n)(x) \ge  \mathscr B^*(f, h)\circ\tau_n(x), \quad \quad \forall x\in X\setminus C^{\mathscr B^{*}(f,h)\circ\tau_n}. 
\]
Combining this with \eqref{eq:sup-infty-on-orbit}, we obtain
\begin{equation}\label{eq:sup-Tstar-gj}
\sup_{n\ge1} \mathscr B^*(f\circ\tau_n, h\circ\tau_n)(x) 
\;=\;
\infty
\quad\text{for all   }x\in X\setminus \left(\mathcal N_0\cup \bigcup_{n=1}^\infty C^{\mathscr B^{*}(f,h)\circ\tau_n} \right).
\end{equation}

Summarizing, we have constructed  sequences $\{f_n\}\subset L^p(X)$, $\{h_n\}\subset L^q(X)$
such that

\begin{enumerate}
  \item $f_n\circ\tau_n(x) \ge 0$, $h_n\circ\tau_n(x) \ge 0$   for every $n$;
  \item \(\displaystyle
         \sum_{n=1}^\infty \|f_n\circ\tau_n\|_{L^p}^p < \infty,\quad  \sum_{n=1}^\infty \|h_n\circ\tau_n\|_{L^q}^q < \infty;
        \)
  \item \(\displaystyle
         \sup_{n\ge1} \mathscr B^*(f_n\circ\tau_n,h_n\circ\tau_n)(x) = \infty
        \quad\text{for all   }x\in X\setminus \left(\mathcal N_0\cup \bigcup_{n=1}^\infty C^{\mathscr B^{*}(f,h)\circ\tau_n} \right).
        \)
\end{enumerate}

Now define functions on $X$ by setting
\[
F(x)
:=
\Biggl(\sum_{n=1}^\infty f_n\circ\tau_n(x)^p\Biggr)^{1/p},
\]
\[
H(x)
:=
\Biggl(\sum_{n=1}^\infty h_n\circ\tau_n(x)^q\Biggr)^{1/q}.
\]
By Tonelli's theorem, the function $F$ lies in
$L^p(X)$ as
\[
\|F\|_{L^p}^p
=
\int_X \sum_{n=1}^\infty \big(f_n\circ\tau_n(x)\big)^p\,d\mu(x)
=
\sum_{n=1}^\infty \|f_n\circ\tau_n\|_{L^p}^p
< \infty.
\]
Likewise, the function $H$ lies in
$L^q(X)$.

Next, for each $n$ and each $x$, since $f_n\circ\tau_n,h_n\circ\tau_n\ge 0$,  we have
\[
\mathscr B_n(f_n\circ\tau_n,h_n\circ\tau_n)(x)\ge 0 \quad\textup{for all } x \in X\setminus \mathcal N_n, 
\]
where $\mathcal N_n$ is a set of measure zero. 
Using the Lemma~\ref{abslo}, we obtain 
\[
\mathscr B^*(F,H)(x)
=
\sup_n |\mathscr B_n(F,H)(x)|
\;\ge\;
\sup_n \mathscr B_n (f_n\circ\tau_n,h_n\circ\tau_n)(x)
\;\ge\;
\mathscr B^*(f_n\circ\tau_n,h_n\circ\tau_n)(x)
\]
for $x$ in $X\setminus \mathcal M_n$, where $\mathcal M_n$ is a null subset of $X$. 
Thus, by taking $\mathcal C:=\mathcal N_0\cup \bigcup_{n=1}^\infty C^{\mathscr B^{*}(f,h)\circ\tau_n}$,  and $\mathcal M:=\bigcup_{n=1}^\infty \mathcal M_n$, we have
\[
\mathscr B^*(F,H)(x) \;\ge\; \mathscr B^*(f_n\circ\tau_n,h_n\circ\tau_n)(x)
\quad\text{for all $n$ and all $x\in X\setminus (\mathcal C\cup \mathcal M)$}
\]
so that
\[
\mathscr B^*(F,H)(x)
\;\ge\;
\sup_{n\ge1} \mathscr B^*(f_n\circ\tau_n,h_n\circ\tau_n)(x)
=
\infty
\quad\text{for almost every }x\in X,
\]
by \eqref{eq:sup-Tstar-gj}. This contradicts the hypothesis that
$\mathscr B^*(F,H)(x)<\infty$ a.e. for   every  $F\in L^p(X)$ and $H\in L^q(X)$.

\medskip

Therefore our assumption that $\mathscr B^*$ is not of weak-type $(p,q,r)$
must have been false. Thus $\mathscr B^*$ must satisfy the   weak-type $(p,q,r)$ condition \eqref{pqr44},  and this 
completes the proof.
\end{proof}

\section{An application concerning maximal bilinear averages}\label{APP}

Let $(X,\mathcal{B},\mu)$ be a finite measure space and let
$\tau:X \to X$ be a measure--preserving transformation.  
Then the quadruple $(X,\mathcal{B},\mu,\tau)$ forms a
measure--preserving dynamical system.
In this section we study the family of bilinear averages
\[
\mathscr B_m(f,h)(x)
    :=\frac{f(\tau^m x)\, h(\tau^{2m}x)}{|m|+1},
    \qquad m\in\mathbb Z,
\]
defined for $(f,h)\in L^p(X)\times L^q(X)$,
and the associated maximal operator
\[
\mathscr B^*(f,h)(x)
   := \sup_{m\in\mathbb Z}\left|\mathscr B_m(f,h)(x)\right|, \qquad x\in X.
\]
This operator is commonly referred to as the
\emph{bilinear tail maximal operator} and was introduced and studied by
Assani and Buczolich~\cite{Assani2010}. We first recall the following result.

\begin{theorem}[Assani--Buczolich {\cite{Assani2010}}]\label{1p1qle2}
Let $(X,\mathcal{B},\mu,\tau)$ be a measure--preserving dynamical system on a
probability space $X$.  If $p,q\ge1$ satisfy
\[
\frac1p+\frac1q<2,
\]
then the maximal operator
\[
\mathscr B^*(f,h)(x)
    = \sup_{m\in\mathbb Z}
      \left|\frac{f(\tau^m x)\, h(\tau^{2m}x)}{|m|+1}\right|
\]
maps $L^p(X)\times L^q(X)$ into $L^r(X)$ for every exponent
\[
0<r<\frac12.
\]
\end{theorem}

Assuming that $\tau$ is ergodic, we are able to obtain a
substantially stronger conclusion.
More precisely, we show that $\mathscr B^*$ satisfies a strong--type estimate
at the critical exponent $r=(\frac1p+\frac1q)^{-1}$.
This extends the integrability range of $\mathscr B^*(f,h)$
beyond that obtained in \cite{Assani2010}.

The ergodicity hypothesis is mild in this context, yet essential:
it allows the family $\{\mathscr B_m\}$ to be placed within
Sawyer’s abstract framework for maximal operators generated by
ergodic families of transformations, while all other structural
features of the bilinear averages remain unchanged.

We now state our main results for this application.
\begin{theorem}\label{App1}
Let $(X,\mathcal{B},\mu,\tau)$ be a measure--preserving dynamical system on a
probability space $X$, and suppose that $\tau$ is ergodic.
Let $1< p,q\le\infty$ and $\frac12< r\le\infty$ satisfy
$\frac1r=\frac1p+\frac1q$.
Then there exists a constant $C>0$ such that for all
$f\in L^p(X)$, $h\in L^q(X)$ we have 
\begin{equation}\label{2324}
\big\| \mathscr B^*(f,h)\big\|_{L^r(X)} C
      \|f\|_{L^p(X)} \,\|h\|_{L^q(X)} .
\end{equation}
\end{theorem}

\begin{theorem}\label{App1_endpoints}
Let $(X,\mathcal{B},\mu,\tau)$ be a measure--preserving dynamical system on a
probability space $X$, and suppose that $\tau$ is ergodic. Fix $p,q>1$ and 
let $r$ satisfy $1/r=1/p+1/q$.
Then there exists a constant $C>0$ such that for all $\alpha>0$ we have 
\[
  \mu\Bigl(\bigl\{x\in X:\mathscr B^*(f,h)(x)>\alpha\bigr\}\Bigr)
  \le \frac{C}{\alpha^r}\,\|f\|_{L^1(X)}^r\,\|h\|_{L^q(X)}^r
\]
for all $f\in L^1(X)$ and $h\in L^q(X)$.  We also   have
\[
  \mu\Bigl(\bigl\{x\in X:\mathscr B^*(f,h)(x)>\alpha\bigr\}\Bigr)
  \le \frac{C}{\alpha^r}\,\|f\|_{L^p(X)}^r\,\|h\|_{L^1(X)}^r
\]
for all $f\in L^p(X)$ and $h\in L^1(X)$.

\end{theorem}

We obtain our theorem as a consequence of   
Theorem~\ref{thm2}, but in order to apply the latter, we will need   
the following two lemmas.

\begin{lemma}\label{APP1}
For each $m\ge1$, the positive bilinear operator
$\mathscr B_m:L^p(X)\times L^q(X) \to\mathcal M(X)$
is continuous-in-measure in each variable.
\end{lemma}

\begin{proof}
Let $f_k\to f$ in $L^p(X)$ and $h_k\to h$ in $L^q(X)$.
Since $\tau$ is measure preserving, composition with $\tau^m$ and
$\tau^{2m}$ preserves $L^p$ and $L^q$ norms, hence
\[
f_k\circ\tau^m \to f\circ\tau^m \quad\text{in } L^p(X),
\qquad
h_k\circ\tau^{2m} \to h\circ\tau^{2m} \quad\text{in } L^q(X).
\]

Set $r>0$ by $1/r=1/p+1/q$.
We express the difference as follows: 
\begin{align*}
\mathscr B_m(f_k,h_k)-\mathscr B_m(f,h)
 =\frac{1}{|m|+1}\Bigl[
(f_k-f)\circ\tau^m \cdot h_k\circ\tau^{2m} 
+ f\circ\tau^m \cdot (h_k-h)\circ\tau^{2m}
\Bigr].
\end{align*}
By H\"older's inequality,
\[
\|(f_k-f)\circ\tau^m \cdot h_k\circ\tau^{2m}\|_{L^r}
\le \|f_k-f\|_{L^p}\,\|h_k\|_{L^q},
\]
and
\[
\|f\circ\tau^m \cdot (h_k-h)\circ\tau^{2m}\|_{L^r}
\le \|f\|_{L^p}\,\|h_k-h\|_{L^q}.
\]
Since $f_k\to f$ in $L^p(X)$ and $h_k\to h$ in $L^q(X)$, the right-hand
sides tend to zero as $k\to \infty$. Consequently,
\[
\|\mathscr B_m(f_k,h_k)-\mathscr B_m(f,h)\|_{L^r}\longrightarrow 0.
\]

Finally, Chebyshev's inequality yields, for every $\varepsilon>0$,
\[
\mu\Bigl(
\bigl|\mathscr B_m(f_k,h_k)-\mathscr B_m(f,h)\bigr|>\varepsilon
\Bigr)
\le
\varepsilon^{-r}
\|\mathscr B_m(f_k,h_k)-\mathscr B_m(f,h)\|_{L^r}^{\,r}
\longrightarrow 0,
\]
as $k\to \infty$, which is exactly convergence in measure.
\end{proof}

\begin{lemma}\label{APP2}
Let $\mathcal T=\{\tau^k:k\in\mathbb Z\}$.
If $\tau$ is ergodic, then $\mathcal T$ is an ergodic family of
measure--preserving transformations.
Moreover, the bilinear family $\{\mathscr B_m\}_{m\ge1}$ is distributive with
respect to $\mathcal T$, in the sense that
\[
\mathscr B_m(f,h)(\tau^k x)
=
\mathscr B_m(f\circ\tau^k,\; h\circ\tau^k)(x)
\]
for all integers $m,k$ and all $x\in X$.
\end{lemma}

\begin{proof}
By definition, a family of measure--preserving transformations is ergodic
if the only measurable sets invariant under every transformation in the
family have measure $0$ or $1$.

Since invariance under all $\tau^k$ is equivalent to invariance under
$\tau$, the ergodicity of $\tau$ implies that $\mathcal T$ is ergodic.

The distributive identity follows from a direct computation:
\[
\mathscr B_m(f,h)(\tau^k x)
= \frac{f(\tau^{m+k}x)\,h(\tau^{2m+k}x)}{|m|+1}
= \mathscr B_m(f\circ\tau^k,\;h\circ\tau^k)(x).
\]
This completes the proof.
\end{proof}

 Having verified these lemmas,  we proceed with the proofs of    
Theorems~\ref{App1} and~\ref{App1_endpoints}. 

\begin{proof}
Lemmas~\ref{APP1} and~\ref{APP2} assert exactly the continuity and 
distributivity required by  
Theorem~\ref{thm2}. 
Next we need to know that $\mathscr B^*(f,h)$ is finite a.e. for $f,h$ in the given spaces. 
We obtain these assertions appealing to  Theorems 1  and 2   in 
 Assani and Buczolich~\cite{Assani2010}. These theorems claim that when $p,q>1$, 
 for $f\in L^p$ and $h\in L^q$ we have that  $\mathscr B^*(f,h)$ lies in $L^r$ for $r<1/2$ (\cite[Theorem 1]{Assani2010}), while if $p>1$ and $f\in L^p$ and $g\in L^1$,  then $\mathscr B^*(f,h)$ lies in $L^{1/2,\infty}$   (\cite[Theorem 2]{Assani2010}). In both cases we have $\mathscr B^*(f,h)<\infty$ a.e., 
 so hypothesis \eqref{kkdjkjd} of Theorem~\ref{thm2} is valid.

Applying Theorem~\ref{thm2} we obtain the weak type result:  for all $f\in L^p(X)$ and $h\in L^q(X)$ we have 
\[
  \mu\Bigl(\bigl\{x\in X:\mathscr B^*(f,h)(x)>\alpha\bigr\}\Bigr)
  \le \frac{C}{\alpha^r}\,\|f\|_{L^p(X)}^r\,\|h\|_{L^q(X)}^r
\]
when $1/p+1/q=1/r$ and $p,q\ge 1$. Note that if   $p=q=\infty$, then the above estimate is indeed a 
strong type estimate.    If both $p,q>1$, 
then   the weak--type estimate~\eqref{2324} can be upgraded to
a strong--type bound
\[
\|\mathscr B^*(f,h)\|_{L^r}
\le C\,\|f\|_{L^p}\,\|h\|_{L^q},
\]
by bilinear interpolation; see \cite{GLLZ}. Bilinear interpolation can be applied 
since we have an open convex 
region on which weak type bounds are valid. 
In the second  case where at least one but not both of $p,q$ is equal to $1$, 
Theorem~\ref{thm2} yields exactly the claimed weak--type bound in Theorem~\ref{App1_endpoints}.
\end{proof}

\section{An application concerning maximal bilinear Bochner--Riesz operators}\label{br}
The Bochner--Riesz means present natural ways to obtain the  summability and
pointwise convergence of Fourier series and Fourier integrals in higher dimensions. 
The study of bilinear Bochner--Riesz means
 have attracted  attention in recent years, beginning
with the  works \cite{BernicotSongYan2015,JeongLeeVargas2018}, where
$L^p\times L^q \to L^r$ boundedness is 
investigated.

In the bilinear setting, for $\alpha \ge 0$ and $\lambda>0$, the bilinear
Bochner--Riesz operator on $\mathbb R^n$ is defined by
\[
\mathcal B^\alpha_\lambda(f,g)(x)
:=
\iint_{\mathbb R^n\times\mathbb R^n}
\bigl(1- \lambda^{-2}(|\xi|^2+|\eta|^2)\bigr)_+^{\alpha}\,
\widehat f(\xi)\,\widehat g(\eta)\,
e^{2\pi i x\cdot(\xi+\eta)}\,d\xi\,d\eta.
\]

Motivated by questions of pointwise convergence and almost everywhere control,
maximal variants of bilinear Bochner--Riesz operators were subsequently
introduced and studied in
\cite{He2016,JeongLee2020,JotsaroopShrivastava2022}, where necessary and sufficient
conditions for boundedness were obtained in various regimes of indices. In particular,
He~\cite[Proposition 4.1]{He2016} established necessary conditions for the boundedness of the
maximal bilinear Bochner--Riesz operator, highlighting fundamental obstructions
that do not appear in the linear theory.

We now turn to the periodic setting. Let $f$ and $g$ be trigonometric polynomials
on $\mathbb T^n$ with Fourier coefficients $\widehat f(k)$ and
$\widehat g(\ell)$. For $\alpha \ge 0$ and $m \in \mathbb N$, we define the
bilinear Bochner--Riesz means by
\[
\mathcal B^\alpha_m(f,g)(x)
=
\sum_{k,\ell\in\mathbb Z^n}
\Bigl(1-\frac{|k|^2+|\ell|^2}{m^2}\Bigr)_+^{\alpha}
\widehat f(k)\widehat g(\ell)
e^{2\pi i (k+\ell)\cdot x}.
\]

The associated maximal bilinear Bochner--Riesz operator is then given by
\[
\mathcal B^\alpha_*(f,g)(x)
:=
\sup_{m\in\mathbb N}\bigl|\mathcal B^\alpha_m(f,g)(x)\bigr|, \qquad x\in \mathbb T^n.
\]

Although $L^p\times L^q\to L^r$ boundedness for $\mathcal B^\alpha_*$ is known 
for certain  indices $p,q$, we are interested here on the range of indices where 
boundedness fails. A necessary condition for the boundedness of the bilinear 
maximal bilinear Bochner--Riesz operator is given below.
 
\begin{proposition}[He~\cite{He2016}]
\label{prop:necessary_bilinear_BR}
A necessary condition for the bilinear maximal Bochner--Riesz operator
$\mathcal B_*^\alpha$ to be bounded from
$L^{p}(\mathbb{R}^n) \times L^{q}(\mathbb{R}^n)$
to weak $L^{r}(\mathbb{R}^n)$
is that
\[
\alpha \ge \frac{2n-1}{2r} - \frac{2n-1}{2}.
\]
\end{proposition}

\noindent{\bf Remark.}
This result does not present a restriction on $\alpha$ when $r\ge 1$; however, when $r<1$, it 
says that that if  
\[
\mathcal B_*^\alpha: L^p(\mathbb R^n)\times L^q(\mathbb R^n)\to L^{r,\infty}(\mathbb R^n)
\]
then $\alpha$ must satisfy
\[
\alpha \ge \frac{2n-1}{2}\Big(\frac1r-1\Big).
\]

 It is not hard to see, either by considering the same examples on the $n$-torus, or even by 
a bilinear transference principle for maximal operators \cite[Theorem 2]{GrafakosHonzik} that 
Proposition~\ref{prop:necessary_bilinear_BR} also holds when $\mathbb R^n$ is replaced by the 
$n$-torus 
$\mathbb T^n$. Using this fact and Theorem~\ref{thm1} we obtain the following result.

\begin{theorem}
\label{thm:failure-ae-br}
Let $n>1$, $1\le p,q<\infty$,    $\alpha\ge 0$ and suppose $r=(1/p+1/q)^{-1}<1$. Define
\[
\mathcal S^\alpha
:=
\left\{
(p,q,r)\in(1,2]\times(1,2]\times\big(\tfrac12,1\big)
\,:\,
\frac1r=\frac1p+\frac1q
\ \text{and}\
\alpha<\frac{2n-1}{2}\Big(\frac1r-1\Big)
\right\}.
\]

Then if $(p,q,r)\in\mathcal S^\alpha$, then there exist functions $f\in L^p(\mathbb T^n)$ and $g\in L^q(\mathbb T^n)$ such that the limit
\[
\lim_{m\to\infty}\mathcal B^\alpha_m(f,g)(x)
\]
does not exist on a set of positive measure in\/ $\mathbb T^n$.
\end{theorem}

\begin{proof}
We apply Theorem~\ref{thm1} with $M=G=\mathbb T^n$. The group $G$ acts on the space $M$ in terms of the periodic translations
Let $(\tau_y f)(x):=f(x-y)$ for  $y\in\mathbb T^n$.
Then the Fourier coefficients satisfy
\[
\widehat{\tau_y f}(k)=e^{-2\pi i k\cdot y}\,\hat f(k),\qquad
\widehat{\tau_y g}(\ell)=e^{-2\pi i \ell\cdot y}\,\hat g(\ell).
\]
A straightforward calculation yields: for $x\in \mathbb T^n$ we have 
\begin{align*}
\mathcal B^\alpha_m(\tau_y f,\tau_y g)(x)
&=\sum_{k,\ell\in\mathbb Z^n}
\Bigl(1-\frac{|k|^2+|\ell|^2}{m^2}\Bigr)_+^{\alpha}
\widehat{\tau_y f}(k)\widehat{\tau_y g}(\ell)\,
e^{2\pi i (k+\ell)\cdot x}\\
&=\sum_{k,\ell\in\mathbb Z^n}
\Bigl(1-\frac{|k|^2+|\ell|^2}{m^2}\Bigr)_+^{\alpha}
e^{-2\pi i k\cdot y}\hat f(k)\,e^{-2\pi i \ell\cdot y}\hat g(\ell)\,
e^{2\pi i (k+\ell)\cdot x}\\
&=\sum_{k,\ell\in\mathbb Z^n}
\Bigl(1-\frac{|k|^2+|\ell|^2}{m^2}\Bigr)_+^{\alpha}
\hat f(k)\hat g(\ell)\,
e^{2\pi i (k+\ell)\cdot (x-y)}\\
&=\mathcal B^\alpha_m(f,g)(x-y)
=\tau_y \mathcal B^\alpha_m(f,g)(x).
\end{align*}
Hence $\mathcal B^\alpha_m$ is distributive  in accordance with   Definition~\ref{Def8888}.
 
Suppose, contrary to the claim, that for every $f\in L^p(\mathbb T^n)$ and
$g\in L^q(\mathbb T^n)$ the limit $\lim_{m\to\infty}\mathcal B^\alpha_m$ exists for almost every
$x\in\mathbb T^n$. Under this assumption, Theorem~\ref{thm1} applies with $M=\mathbb T^n$ and
yields the weak-type estimate
\[
\mu\!\left(\bigl\{x\in\mathbb T^n:\ \mathcal B^\alpha_*(f,g)(x)>\alpha\bigr\}\right)
\le
\frac{C}{\alpha^r}\,\|f\|_{L^p(\mathbb T^n)}^{\,r}\,\|g\|_{L^q(\mathbb T^n)}^{\,r},
\qquad \alpha>0,
\]
for all $f\in L^p(\mathbb T^n)$ and $g\in L^q(\mathbb T^n)$.

In particular, $B^\alpha_*(f,g)$ is of weak type $(p,q;r)$, and hence the family
$\{B^\alpha_m(f,g)\}_{m\in\mathbb N}$ is uniformly bounded from $L^p(\mathbb T^n)\times L^q(\mathbb T^n)$ to
$L^{r,\infty}(\mathbb T^n)$.  
But this would contradict the assertion of Proposition~\ref{prop:necessary_bilinear_BR} when 
$\mathbb R^n$ is replaced by $\mathbb T^n$.

 Therefore,
the assumption of almost everywhere convergence must be false, and there must exist
$f\in L^p(\mathbb T^n)$ and $g\in L^q(\mathbb T^n)$ for which the sequence
$\{\mathcal B^\alpha_m(f,g)\}_{m\in\mathbb N}$ fails to converge almost everywhere.
\end{proof}

\section{An application in differentiation theory}
We apply our results to a bilinear maximal operator generated by dyadic averages on 
the two-dimensional torus $\mathbb T^2$. This result is local and can easily be extended to $\mathbb R^2$; the only reason we work on the torus is to make use of the compactness of the ambient 
space required by our theorems. 
Although the kernels involved are bounded and compactly supported, the resulting maximal operator fails to satisfy weak $L^{\frac{1}{2}}$ estimates. This obstruction precludes a.e.  convergence for the associated bilinear averages, by Theorem~\ref{thm3}. 

Let $G=M=\mathbb T^2=[0,1)^2$ equipped with   normalized Lebesgue measure $\omega_G=\mu$.
For $k \in\mathbb N$, let 
$I_k=[0,2^{-k}]$   be the dyadic interval of length $2^{-k}$ whose left endpoint is $0$.  

For   $(k,\ell) \in \mathbb N\times \mathbb N$ define the kernels
\[
K_{k,\ell}(u,v)
:= \bigg( \frac{\mathbf 1_{I_k\times I_\ell }(u)}{|I_k\times I_\ell |}  - 
\frac{\mathbf 1_{ (I_k\cap I_\ell )\times (I_k\cap I_\ell)}(u) }{|(I_k\cap I_\ell )\times (I_k\cap I_\ell )|} \bigg)
\bigg( \frac{\mathbf 1_{I_k\times I_\ell }(v)}{|I_k\times I_\ell |}  - 
\frac{\mathbf 1_{ (I_k\cap I_\ell )\times (I_k\cap I_\ell)}(v) }{|(I_k\cap I_\ell )\times (I_k\cap I_\ell )|} \bigg)
\]
for $  (u,v)\in \mathbb T^2\times \mathbb T^2.$ We note that the functions  
$\frac{\mathbf 1_{I_k\times I_\ell }(u)}{|I_k\times I_\ell |}\frac{\mathbf 1_{I_k\times I_\ell }(v)}{|I_k\times I_\ell |}$ play a fundamental role in biparameter harmonic analysis and were 
first studied in \cite{GLTP} as kernels of bilinear operators.
Then $K_{k,\ell}\in L^\infty(\mathbb T^2\times \mathbb T^2)$ and
\[
\|K_{k,\ell}\|_{L^\infty}\le \bigg(\frac{1}{|I_k | | I_{\ell}|} +  \frac{1}{|I_k\cap I_\ell |^2   } \bigg)^2<\infty.
\]

Now for $f,h\in L^1(\mathbb T^2)$, define
  the translation--invariant  bilinear operators
\[
T_{k,\ell}(f,h)(x)
:=\iint_{\mathbb T^2\times \mathbb T^2}
K_{k,\ell}(x-y,\,x-z)\,
f(y)\,h(z)\,dy\,dz ,
\]
which can also be written as  
\begin{align*}
&\left(
 \frac{1}{|I_k\times I_{\ell}| }  \int_{ I_k\times I_{\ell}} f(x-y)\,dy
 -  \frac{1}{| (I_k\cap I_\ell ) |^2 } \int_{ (I_k\cap I_\ell )\times (I_k\cap I_\ell )} f(x-y)\,dy
 \right) \\
&\hspace{1in}\left(
 \frac{1}{|I_k\times I_{\ell}| } \int_{ I_k\times I_{\ell}} h(x-z)\,dz
 -  \frac{1}{| (I_k\cap I_\ell ) |^2 } \int_{ (I_k\cap I_\ell )\times (I_k\cap I_\ell )} h(x-z)\,dz
 \right) . 
\end{align*} 
Finally, we consider  the associated maximal operator
\[
T^*(f,h)(x)
=\sup_{k,\ell\ge 1}\,|T_{k,\ell}(f,h)(x)|
\]
defined for integrable functions on the torus. 
 
\begin{proposition}\label{App3}
    Then there exist functions $f_0 ,h_0 \in L^1(\mathbb T^2)$, such that   
    $$
    \limsup_{\min (k,\ell) \to \infty}| T_{k,\ell}(f_0,h_0)(x)|=\infty
    $$ 
    on a set of positive measure.
\end{proposition}

\begin{proof}
We start by showing that $T^*$ fails to map $L^1(\mathbb T^2)\times L^1(\mathbb T^2)$ into
$L^{1/2,\infty}(\mathbb T^2)$.
Notice that $T_{k,\ell}$ vanishes when $k=\ell$, so we may assume that $k\neq\ell$. 

We consider the Dirac mass $\delta_0$ at the origin and the action of $T_{k,\ell}$ on the pair $(\delta_0,\delta_0)$. 
We work with the subset of $\mathbb T^2$ given by the set of all $(x_1,x_2)$ with 
$$
0<   2x_1<x_2<1.
$$
Pick $k_0$ and $\ell_0$ natural numbers such that 
$$
2^{-k_0} \le x_1<2^{-k_0+1}, \qquad 2^{-\ell_0} \le x_2<2^{-\ell_0+1}.
$$
Then 
\[
T^*(\delta_0,\delta_0)(x_1,x_2)\ge  (2^{k_0+\ell_0-2})^2\ge   \frac{2^{-4}}{(x_1x_2)^2}.
\]

Let
\[
\Omega_\lambda
:=\Bigl\{(x_1,x_2)\in [0,1)^2:\ 0<2x_1<x_2<1,\ \ x_1x_2<(2\lambda)^{-1/2}\Bigr\},
\qquad \lambda>2.
\]
For   fixed \(x_2\in(0,1)\), the constraints give
\[
0<x_1<\min\Bigl(\frac{x_2}{2},\frac{(2\lambda)^{-1/2}}{x_2}\Bigr), 
\]
hence the measure of \(\Omega_\lambda\) is
\[
|\Omega_\lambda|
=\int_{0}^{1 }\min\Bigl(\frac{x_2}{2},\frac{(2\lambda)^{-1/2}}{x_2}\Bigr)\,dx_2.
\]
Note that 
\[
\frac{x_2}{2}\ge\frac{(2\lambda)^{-1/2}}{x_2}
\quad\Longleftrightarrow\quad
x_2 \ge (2/\lambda)^{1/4}.
\]
 
Now since \( \lambda > 2\) we have $(2/\lambda)^{1/4}<1 $, and so
\[
|\Omega_\lambda|
 \ge  \int_{(2/\lambda)^{1/4}}^{1 }\frac{(2\lambda)^{-1/2}}{x_2}\,dx_2 \\
 = \frac{1}{4 \sqrt{2}} \frac{1}{\sqrt{\lambda}} \ln \Big( \frac {\lambda}{2} \Big).
\]
Consequently, 
\[
\lambda^{1/2}|\Omega_\lambda|
\ge\frac{1}{8}\ln \Big( \frac {\lambda}{2} \Big)\to \infty
\qquad\text{as }\lambda\to\infty.
\]

This gives
\[
\big\| T^*(\delta_0,\delta_0)\|_{L^{1/2,\infty}}=\sup_{\lambda>0}\ \lambda^{1/2}
\big|\{x\in\mathbb T^2:\ T^*(\delta_0,\delta_0)(x)>\lambda\}\big|
\ge \sup_{\lambda>2} \lambda^{1/2}|\Omega_\lambda|
=\infty,
\]
and shows that
\[
T^*(\delta_0,\delta_0)\notin L^{1/2,\infty}(\mathbb T^2).
\]

Then by Theorem~\ref{thm3},  there exist $f_0,h_0\in L^1(\mathbb T^2)$ such that 
$$
\limsup_{\min(k,\ell) \to \infty} T_{k,\ell}(f_0,h_0)=\infty
$$
 on  a set of positive measure. This concludes the proof of Proposition~\ref{App3}. 
\end{proof}

One concludes that the Lebesgue differentiation property
\begin{equation}\label{LD99P}
\lim_{\min (k,\ell) \to \infty}T_{k,\ell}(f ,h ) = 0 \qquad \textup{a.e.}
\end{equation}
fails on $L^1\times L^1$. The operator $T^*$ in Proposition~\ref{App3} is controlled by 
  the bilinear strong maximal function studied in \cite{GLTP}, which is bounded 
  from $L^p\times L^q \to L^r$, when $1<p,q \le \infty$ and $1/r=1/p+1/q$.  In the  case where $p,q>1$, 
  by standard arguments, the Lebesgue differentiation property \eqref{LD99P} holds.

Naturally, there is nothing special about the case $d=2$ other than the 
simplicity of the notation. In fact, in Proposition~\ref{App3}, 
the space $\mathbb T^2$ can be replaced by $\mathbb T^d$ for any $d\ge 2$.

\end{document}